\documentclass[11pt]{amsart}
\usepackage{amsmath}
\usepackage{amssymb}
\usepackage{amsfonts}
\usepackage{amsthm}
\usepackage{verbatim}
\usepackage{amscd}
\usepackage{cite}
\usepackage{leftidx}
\usepackage{enumerate}
\usepackage{txfonts}
\usepackage{manfnt}
\usepackage{amscd}
\usepackage[mathscr]{eucal}
\usepackage{hyperref}
\usepackage{bookman}
\def\pipi{\pi/\pi}
\def\de{\mathcal{DE}}
\def\deb{\mathcal{DE}^\bullet}
\def\ed{\mathcal{ED}}
\def\edb{\mathcal{ED}^\bullet}
\def\dbar{\bar{\del}}

\newtheorem{thm}{Theorem}  [section]

\newtheorem*{thm*}{Theorem}
\newtheorem*{prop*}{Proposition}
\newtheorem{cor}[thm]{Corollary}
\newtheorem*{cor*}{Corollary}

\newtheorem{lem}[thm]{Lemma}

\newtheorem*{claim*}{Claim}
\newtheorem{prop}[thm]{Proposition}

\theoremstyle{remark}

\newtheorem{rem}[thm]{Remark}
\newtheorem*{rem*}{Remark}
\newtheorem{crit-rem}[thm]{Critical remark}
\newtheorem{remarks}[thm]{Remarks}

\newtheorem{example}[thm]{Example}
\newtheorem*{example*}{Example}

\newtheorem*{defn*}{Definition}

 \DeclareMathOperator{\gm}{\mathbb G_m}

\def\inv{^{-1}}

\def\alb{\mathrm{alb}}

\DeclareMathOperator{\del}{\partial}
 \DeclareMathOperator{\Hom}{Hom}

\DeclareMathOperator{\dlog}{dlog}

\def\calh{\mathcal {H}}
\def\cF{\mathcal{F}}

\def\refp #1.{(\ref{#1})}
\newcommand\carets [1]{\langle #1 \rangle}
\newcommand\caretsmx[1]{\langle #1 \rangle_\mathrm{mx}}

\def\bfm{\mathbf m}

\def\sN{\mathcal N}

\newcommand{\ul}[1]{\underline {#1}}

\def\sbr #1.{^{[#1]}}
\def\sfl #1.{^{\lfloor #1\rfloor}}

\def\what{\widehat}
\def\inv{^{-1}}
\def\?{{\bf{??}}}

\def\dgla{differential graded Lie algebra\ }

\def\H{\mathcal H}

\def\HH{\mathbb H}

\def\C{\mathbb C}
\def\P{\mathbb P}
\def\N{\mathbb N}
\def\R{\mathbb R}

\def\uk{{\ul{k}}}

\def\sgn{\text{\rm sgn} }

\def\Q{\mathbb Q}

\def\O{\mathcal O}

\def\id{\text{id}}

\def\g{\mathfrak g}

\def\1/2{\frac{1}{2}}

\def\I{\mathcal{ I}}

\def\im{\text{im}}

\def\2{{[2]}}
\def\l{\ell}
\def\nl{\newline}

\def\<{\langle}
\def\>{\rangle}

\def\im{\text{im}}
\def\2{{[2]}}
\def\l{\ell}

\def\scl #1.{^{\lceil#1\rceil}}
\def\spr #1.{^{(#1)}}
\def\sbc #1.{^{\{#1\}}}

\def\subpr#1.{_{(#1)}}
\def\ds{\vskip 1cm}

\def\beq{\begin{equation*}}
\def\eeq{\end{equation*}}

\newcommand{\llog}[1]{\langle\log {#1} \rangle}

\newcommand{\llogm}[1]{\langle\log  {^-#1} \rangle}
\newcommand{\llogs}[1]{\langle\log  {^*#1} \rangle}
\newcommand{\llogmp}[1]{\langle\log  {^{\mp}#1} \rangle}
\newcommand{\llogmps}[2]{\langle\log  {^{\mp#2}#1} \rangle}

\newcommand{\mlog}[1]{\langle-\log {#1} \rangle}
\def\g3{{\Gamma\spr 3.}}

\newcommand{\eqspl}[2]{
\begin{equation}\label{#1}
\begin{split}
#2\end{split}\end{equation}}
\newcommand{\eqsp}[1]{\begin{equation*}
\begin{split}#1\end{split}\end{equation*}}

\newcommand{\exseq}[3]{
0\to #1\to #2\to #3\to 0
}

\newcommand{\beginalphaenum}{
\begin{enumerate}\renewcommand{\labelenumi}{ }
\item \begin{enumerate}
}

\def\eex{\end{rm}\end{example}}

\begin{document} 

\title{Differential complexes and Hodge theory on log-symplectic manifolds}
\author 
{Ziv Ran}


\thanks{arxiv.org/1710.11179 }
\date {\today}


\address {\nl UC Math Dept. \nl
Big Springs Road Surge Facility
\nl
Riverside CA 92521 US\nl 
ziv.ran @  ucr.edu\nl
\url{http://math.ucr.edu/~ziv/}
}


\begin{abstract}
We study certain complexes of differential forms,
including 'reverse de Rham'  complexes, on 
(real or complex) Poisson manifolds, especially
holomorphic  log-symplectic ones.
We relate these to the degeneracy divisor and rank loci of the Poisson bivector.
In some good holomorphic cases we compute the local cohomology
of these complexes. In the K\"ahlerian case, we deduce a relation
between the multiplicity
 loci of the degeneracy divisor  and the Hodge numbers of the manifold.
 We also show that vanishing of one of these Hodge numbers is related to
 unobstructed deformations of the normalized degeneracy divisor with
 its induced Poisson structure.
\end{abstract}
\maketitle
\small
\tableofcontents
\normalsize
\section*{Introduction}
One of the interesting features of geometry on (real or complex) 
Poisson manifolds $(X, \Pi)$ is the richness
of the calculus, which in a sense is twice as rich as on a plain manifold:
the usual plus a dual. Interesting
differential operators can be constructed using the Poisson bivector $\Pi$. One of these
is the Koszul-Brylinski operator on differential forms:
\[\del=d\iota_\Pi-\iota_\Pi d\]
where $d$ is exterior derivative and $\iota_\Pi$ denotes interior 
multiplication by $\Pi$.
This is an operator of degree (-1) on differential forms, and Brylinski \cite{brylinski}
has shown that it has square zero, hence gives rise to a 'reverse de Rham' complex:
\[...\stackrel{\del}{\to}\Omega^i_X\stackrel{\del}{\to}\Omega_X^{i-1}\stackrel{\del}{\to}...\]
He has also shown, using the Hodge $*$ operator in the real $C^\infty$ category,  
that when $\Pi$ is a symplectic Poisson structure, i.e. everywhere nondegenerate,
the reverse de Rham complex is 
equivalent to the usual de Rham complex, hence computes
the cohomology $H^\bullet(X, \R)$.\par
Here we start with the observation that a different set of operators
of degree (-1), namely
\[\delta_i=id\iota_\Pi-(i-1)\iota_\Pi d:\Omega^{n+i}_X\to \Omega_X^{n+i-1}\]
($n\in\N$  fixed, usually as half the dimension of $X$),
can be used to construct a 
reverse De Rham complex $\Theta^\bullet=\Omega_X^{\dim(X)-\bullet}$ 
of differential forms called the
'Mahr de Poisson' or MdP complex, in either the $C^\infty$
or holomorphic category (or for that matter, in any setting where $d$ and $\iota_\Pi$
make sense). 
More generally, for any $\lambda\in\C$, there is a complex $\Theta^\bullet_\lambda$
with differential
\[\delta_{\lambda, j}=(j+\lambda)d\iota_\Pi-(j+\lambda-1)\iota_\Pi d
:\Omega^j\to\Omega^{j-1}.\]
Note that unlike the De Rham complex, the MdP complex need not be acyclic
locally where $\Pi$ degenerates,  indeed its local cohomology
seems difficult to compute in general. Some special cases 
will be computed below.
\par
A special feature of the MdP complex $\Theta^\bullet_X$,
on a $2n$-dimensional 
($C^\infty$ or holomorphic) Poisson manifold
$(X, \Pi)$, not shared by Brylinski's complex and
 which makes $\Theta^\bullet_X$ amenable to study, is the
existence of 'bonding' maps relating it to the de Rham complex:
\[\pi:\Theta^\bullet_{X, \leq  n}\to \Omega^\bullet_{X, \leq n},\]
\[\pi':\Omega^\bullet_{X, \geq n}\to\Theta^\bullet_{X, \geq  n} .\]

It is also possible to construct  a pair of hybrid complexes on the top or bottom half of the de Rham groups:
\[\ed^\bullet: \Omega^{2n}_X\stackrel{\delta_n}{\to}\Omega^{2n-1}_X...\Omega^{n+i}_X\stackrel{\delta_i}{\to}
\Omega^{n+i-1}_X...\to \Omega_X^{n+1}
\stackrel{\delta_1}{\to}\Omega^n_X\stackrel{d}{\to}\Omega^{n+1}_X...\stackrel{d}{\to}\Omega^{2n}_X,
\]
\[\de^\bullet: 
\O_X\stackrel{d}{\to}\Omega^1_X...\stackrel{d}{\to}\Omega^n_X\stackrel{\delta_0}{\to}\Omega^{n-1}_X...
\Omega^{n-i}_X\stackrel{\delta_{-i}}{\to}\Omega^{n-i-1}_X...\to \Omega_X^1\stackrel{\delta_{-n+1}}{\to}\O_X,
\]
together with  a map of complexes
\[\pi:\ed^\bullet\to\de^\bullet.\]
The mapping cone of  $\pi$   may be thought of as a 'double helix'
with strands $\ed^\bullet$ and $\de^\bullet$ or $\Theta^\bullet$ and $\Omega^\bullet$
\footnote{The bonds between the two strands of the DNA molecule
are called $\pi$ bonds}.\par
In the case where $\Pi$ is  pseudo-symplectic, i.e.
nondegenerate almost everywhere, hence
degenerates along a Pfaffian divisor $P$, these complexes
are closely related  to a (singular) 
codimension-1 'kernel foliation' on $P$ (also called 'symplectic foliation"
in the literature).
In general, this kernel foliation is not  'tame' in the sense that the leaves are Zariski-locally closed
(see Example \ref{toric} below).
In fact, leaves can be dense in $P$.\par
When $X$ is a compact K\"ahler manifold, the cohomology of
the MdP complex admits a Hodge decomposition like its De Rham 
analogue. Indeed the 'Hodge diamond' for $\Theta^\bullet$ is just
a $90^\circ$ turn of the usual.\par
We will concentrate mainly on the case where $\Pi$ is log-symplectic, i.e. the degeneracy
divisor $D=D(\Pi)$ has normal crossings. In that case $\Pi$ corresponds to a log-symplectic form
$\Phi$, i.e.  a closed log 2-form whose polar locus coincides with $D$.
For certain purposes it is easier to work directly with $\Phi$ rather than
$\Pi$.
In the case where the log-symplectic structure $\Pi$ satisfies a 
certain 'residual generality' condition (see \S
\ref{rg-sec}), we will study the image of $\pi$ via the corresponding
log-symplectic form $\Phi$ and consequently we will be able to determine the image of $\pi$
via a simplicial resolution, and hence determine the local cohomology
of $\Theta^{n]}$, i.e. the 'upper half' of $\Theta^\bullet$. Curiously, the
method does not seem to adapt easily to the case of the lower
half $\Theta^{n]}$.\par

For other work on De Rham-like complexes and degeneracy of 
log-symplectic Poisson structures,  see \cite{polishchuk-ag-poisson},
\cite{kontsevich-gen-tt}, \cite{ka-kontsevich-pa},
and \cite{lima-pereira}.
In particular, Polishchuk \cite{polishchuk-ag-poisson} constructs
and analyzes
a different differential complex on a Poisson manifold with normal crossings
Pfaffian divisor.\par
I am grateful to the referee for numerous helpful comments.
\subsection*{Notations and conventions}
We work over $\C$. For a natural number $k$, $\mathbf k$ denotes
$\{1,...,k\}$. For a multi-index $I=(i_1<...<i_r)$, $|I|$ denotes the degree, i.e. $r$.
Note the difference between $"i"$, used for indices, and $"\iota"$, used for inclusion maps.
\section{Preliminaries: twisted log complexes}
In this section we study some differential complexes attached to a general
log pair $(X, D)$, i.e. a complex manifold endowed with a reduced, locally normal-crossing
divisor. No Poisson, symplectic or log-symplectic structure is assumed.
\subsection{Minor log complex and compactly supported cohomology}
Here we study certain twists of the
log complex on a log pair $(X, D)$.
Let $X$ be a complex manifold of dimension $d$
endowed with a divisor $D$ with local
normal crossings.
 We remark that in our 
 subsequent application,
 $X$ will be $X_1$, normalization of the degeneracy divisor $D(\Pi)$ of of a 
 log symplectic manifold $(X, \Pi)$, and $D$ will be the double point locus
 of the map $X_1\to D(\Pi)\subset X$.  \dbend This will result in a shift of indices!!
\vskip .5cm
Via the inclusion $\Omega_X^\bullet\llog{D}\subset\Omega^\bullet_X(D)$,
we get a graded subgroup
\eqspl{}{
\Omega^\bullet_X\llogm{D}:=\Omega_X^\bullet\llog{D}(-D)\subset\Omega_X^\bullet.
}
Locally, letting $F=x_1...x_k$ be an equation for $D$,
$\Omega^\bullet_X\llogm{D}$ is generated by differentials of the following form,
in which $\mathbf k$ denotes $\{1,...,k\}$ and $J=(j_1<...<j_r)$ is a multi-index:
\[\omega_{J, \mathbf k}=\prod\limits_{j\in\mathbf{k}\setminus J}x_j
\prod\limits_{j\in J}dx_j, \forall J\subset\mathbf k, k\leq d.\]
 It is clear from this, or otherwise, that
$(\Omega^\bullet_X\llogm{D}, d, \wedge)$ is a dg algebra over $\Omega^\bullet_X\llog{D}$,
called the \emph{log-minus } or \emph{minor log} complex associated to $D$. Given
the equation $F$ as above, $\Omega^\bullet_X\llogm{D}$
can be identified with
$\Omega^\bullet_X\llog{D}^\circ$ which is  $\Omega^\bullet_X\llog{D}$ with 
twisted differential
\[d^\circ=d+\dlog(F).\]
\par
 
\begin{lem}\label{logm-lem}
 The log-minus complex $\Omega_X^\bullet\llogm{D}$ is 
a resolution of the compact-support direct image
$\C_{U!}:=i_{U!}(\C_U)$ where $i_U:U\to X$ is the inclusion 
of $U=X\setminus D$.
\end{lem}
\begin{proof}
There is a natural map $\C_{U!}\to\O(-D)$ which lifts to a map
$\C_{U!}\to \Omega^\bullet_X\llogm{D}$ and the latter
is clearly a quasi- isomorphism over $U$, so it suffices to prove
that $ \Omega^\bullet_X\llog{D}^\circ$ is exact locally
at every point
of $D$.
To simplify notations we assume $D$ is of maximal multiplicity $k=d$
at the given point;
the general case is a product of a maximal-multiplicity case
and a zero-multiplicity case, and one can use a K\"unneth decomposition.
Then all the sheaves in the log complex decompose into homogeneous
components $S^i_{(m.)}$ indexed by 
exterior degree  $i$ and multi-indices $(m.)$ where $m_i\geq 0$.
That is, each local section is an infinite convergent sum
of homogeneous components.
Note that $x_i$ and $dx_i$ both have degree 1. 
For any multi-index $(m_1,...,m_k)$, we set
\[\chi_{(m.)}=\sum m_idx_i/x_i.\]
Now note that
$d^\circ$ maps $S^i_{(m.)}$ to $S^{i+1}_{(m.)}$ and there, in fact,
is given by multiplication by $\chi_{(m.)}+\chi_{(1.)}$ where
$(1.)=(1,...,1)$. Because the latter form is part of a basis of
$S^1_{(0.)}\subset \Omega^\bullet_X\llog{D}$, 
multiplication by it clearly defines an
exact complex and in fact admits a
'homotopy' operator $\iota_v$ given by interior
multiplication by the log vector field
\[v=\sum x_i\del_{x_i}.\]
This has the property that that the commutator 
$[\iota_v, d^\circ]$ is a nonzero multiple
of the identity on each $S^i_{(m.)}$ term.
Therefore $\Omega^\bullet_X\llog{D}^\circ$ is null-homotopic
and exact.
\end{proof}
\begin{rem}
When $D$ is the exceptional divisor of a resolution of singularities,
the minor log complex seems related to the Du Bois
complex of the singularity, see
\cite{steenbrink-dubois}.
\end{rem}
\subsection{Augmented minor log complex}
\label{logmp-sub-sec}
We shall need an enlargement of the log-minus complex along
the double locus of $D$, called the \emph{augmented minor} or 
\emph{log-minus-plus} complex. Let
\[\nu_i:X_i\to D\subset X\]
be the normalization of the $i$-fold locus of $D$, and let $D_i\subset X_i$ be the
natural normal-crossing divisor on $X_i$, which maps to the $(i+1)$-fold locus of
$D$. 
Also set $U_i=X_i\setminus D_i$. This maps to the set of points of
multiplicity exactly $i$ on $D$.
Note the natural surjective pullback map
for all $i\geq 0$, where $X_0=X$ etc.
\[\nu_i^* :\Omega_{X_i}^\bullet\to \nu_{i+1*}\Omega_{X_{i+1}}^\bullet\]
whose kernel is just $\Omega^\bullet_{X_i}\llogm{D_i}$.
We denote by $Z^{[m}$ the truncation of a complex $Z^\bullet$ below degree $m$
(thus $Z^i=0, i<m$).
Set
\eqsp{K_0=\Omega^\bullet_X\llogm{D},\\
K_1=(\nu_1^*)\inv(\Omega_{X_1}^{[1}\llog{D_1}(-D_1)).}
Thus, $K_1$ is a subcomplex of $\Omega_X^\bullet$ which coincides with
$K_0$ off $X_1$ and which, locally at a point of $X_1$ with branch equation $x_k$, 
where $D$ has equation $F=x_1...x_k$, 
is generated by $\Omega_X^\bullet\llogm{D}$ and by differentials
of degree 1 or more, of the form
\[\omega_{I, \l,\mathbf{k}}=
\prod\limits_{i\in\mathbf{k}\setminus  I, i\neq\l}x_i\prod\limits_{i\in I}dx_i=\dlog{(x)}_IF/x_\l, I\subset\mathbf {k\setminus\{\l\}}.\]
In the general case, assuming $K_i$ is constructed, we construct $K_{i+1}$
by modifying $K_i$ along $X_{i+1}$ for forms of degree $i+1$ or more, i.e.
\[K_{i+1}=(\nu_{i+1}^*)\inv(\Omega_{X_{i+1}}^{[i+1}\llog{D_{i+1}}(-D_{i+1})).\]
Finally set
\eqspl{}{
\Omega^\bullet_X\llogmp{D}=K_{d-2}.
} (note that $K_{d-2}=K_{d-1}$ because 
$\Omega^\bullet_{X_{d-1}}\llog{D_{d-1}}(-D_{d_1})=\Omega^\bullet_{X_{d-1}}$).
By construction, $\Omega^\bullet_X\llogmp{D}$ is a $\Omega^\bullet_X\llog{D}$-
module endowed with an increasing filtration $F_\bullet$
with graded pieces 
\[Gr_i^{F_\bullet}(\Omega^\bullet_X\llogmp{D})=\Omega_{X_i}^{[i}\llogm{D_i}.\]


Locally at a point of $X_r$ with branch equations $x_j, j\in J, |J|=r$,
where $D$ has equation $x_{\mathbf k}$, 
$\Omega^\bullet_X\llogmp{D}$ is generated over $\Omega^\bullet_X\llog{D}$
by differentials of the form
\eqspl{}{
\omega_{I,J, \mathbf k}=\dlog(x)_IF/x_J, I\subset \mathbf k\setminus J 
} as well as differentials on $X_{r-1}$ whose pullback on $X_r$ is of such form.
\par
\par
As in Lemma \ref{logm-lem}, we can compute the cohomology
of the augmented minor log complex:
\begin{lem}\label{logmp-hom-lem}
We have
\eqspl{}{
\H^i(\Omega_X^\bullet\llogmp{D})
=\hat\Omega^i_{U_i!}:=i_{U_i!}(\hat\Omega_{X_i}^{i}),  i\geq 0.
} where $\hat\Omega$ denoted closed forms.\par
%
\end{lem}
\begin{proof}
Use the spectral sequence of the filtered complex with
$E_1^{p,q}=H^{p+q}(Gr_{-p})$, together with Lemma
\ref{logm-lem}. The fact that each  $i$th graded piece has cohomology only in degree $i$ 
ensures that the spectral sequence degenerates at $E_1$.

\end{proof}
As a slight generalization of the log-minusplus complex, we have for any 
$s\geq 0$ a complex $\Omega^\bullet_X\llogmps{D}{s}$ defined as above but with
\[K_{i+1}=(\nu_{i+1}^*)\inv(\Omega_{X_{i+1}}^{[i+1-s}\llog{D_{i+1}}(-D_{i+1})), i+1\geq s.\]
We will need this only for $s=1$ which yields a complex with zeroth term 
$\O_X(-\nu_1(D_1))$
(recall that $D_1$ is a divisor on $X_1$ which maps to a codimension-2 locus on $X$).
Note $\Omega^\bullet_X\llogmps{D}{s}$ admits an increasing filtration with graded
pieces $\Omega^\bullet_{X_i}\llogmp{D_i}, i=0,...,s$.
\subsection{Foliated De Rham complex, log version}
With $(X, D)$ a log pair as above, let $\psi$ be a closed log 1-form,
nowhere vanishing as such.
Then $\psi$ generates an $\Omega^\bullet_X\llog{D}$-submodule 
\eqspl{omega-psi}{
\Omega^\bullet_\psi=\psi \Omega^\bullet_X\llogm{D}\subset  \Omega^\bullet_X\llogm{D}[1]}
This is locally the $\Omega_X^\bullet
\llog{D}$-submodule of $\Omega^\bullet_X$ generated by
$\psi F$. 
Thus, sections of $\Omega^\bullet_\psi$
are of the form $\psi F\gamma$ where $\gamma$ is a log form.
Then
\eqspl{}{
\Omega^\bullet_{X/\psi}:=\Omega^\bullet_X/\Omega^\bullet_\psi[-1]
} is called the \emph{foliated De Rham complex} associated to $\psi$.
The differential on $\Omega^\bullet_\psi$ is given by
\[d(\psi F\alpha)=\psi F(d\alpha+\dlog(F)\alpha).\]
Consequently, $\Omega^\bullet_\psi$ is a quotient of   
$\Omega^\bullet_X\llogm{D}^{d-1]}$
 where $d=\dim(X)$ and $\bullet^{d-1]}$ means truncation
in degrees $>d-1$. Locally, choosing an equation
 $F$ for $D$, we may identify $\Omega^\bullet_X\llogm{D}$
as above with the  complex
denoted $\Omega^\bullet_X\llog{D}^\circ$ which is  $\Omega^\bullet_X\llog{D}$ with differential
\[d^\circ=d+\dlog(F).\]
This is defined locally, depending on the choice of local equation $F$.
The kernel of the natural surjection 
$\Omega^\bullet_X\llogm{D}^{d-1]}\to
\Omega^\bullet_\psi$ consists of the forms divisible by $\psi$, hence
can be identified with $\Omega^\bullet_\psi[-1]^{d-1]}$.

Continuing in this manner, $\Omega^\bullet_\psi$ admits a left resolution of the form
\small
\eqspl{resolution}{
\Omega^\bullet_X\llogm{D}[-d+1]^{d-1]}\to...
...
\to\Omega^\bullet_X\llogm{D}[-1]^{d-1]}
\to\Omega^\bullet_X\llogm{D}^{d-1]}\to\Omega^\bullet_\psi
}

\normalsize
Note that $*[-i]^{d]}=(*^{d-i]})[-i]$.
Set
\[K^i=\ker(d^\circ, \Omega^i_X\llog{D}), i\geq 0.\]
By Lemma \ref{logm-lem}, 
\[K^i\simeq\Omega^{i-1}_X\llog{D}/K^{i-1}, i\geq 1.\]
Locally at a point of $D$, the latter is true for $i=0$ as well,
in the sense that $K^0=0$ while locally at a point of $U$, $K^0=\C$.
Moreover $K^i$ is the unique nonvanishing cohomology
sheaf of $\Omega^\bullet_X\llog{D}[-i]^{d]}$, and occurs as $\H^d$.
We now introduce the following generality hypothesis on our form $\psi$:\par
(*) \emph{For each nonnegative integral multi-index $(m.)$, the log differentials
$\psi$ and $\chi_{(m.)}+\chi_{(1.)}$ are  linearly independent,
i.e. generate a free and cofree subsheaf of $\Omega^1_X\llog{D}$,
locally at every point of multiplicity 2 or more on $D$.}
\par
When $\psi$ is one of the forms $\psi_i$ deduced from a log-symplectic structure,
hypothesis (*) is equivalent to the '1-very general' hypothesis introduced in
\cite{logplus}, Erratum, hence weaker than the Residual Generality
condition in \S \ref{rg-sec}. \par
It is essentially clear that a general log 1-form cannot 
be holomorphically integrated and in 2 or more variables, is not even
proportional to an integrable form. Our aim next is
to generalize this observation.\par
Let let $\O_X^\psi$ denote the sheaf of $\psi$-constant holomorphic
functions, i.e. holomorphic functions $g$ such that $dg\wedge\psi=0$.
Locally at $p\in U$ we can write $\psi=dx, \O^\psi_p=\C\{x\}$ for a
 coordinate $x$. Similarly, locally over $U_1$, the smooth part
 of $D$, we can write $\psi=dx/x, \O^\psi_p=\C\{x\}$.
 Also let $U_j\subset X$ denote the set of points where $D$ has
 multiplicity $\leq j$. More generally, we let $U_{i,j}\subset X_i$ denote
 the set of points where $D_i$ has multiplicity $\leq j-i$
 , i.e. the inverse image of the
  set of point in $X$ where the multiplicity of $D$ is in $[i,j]$. 
  Thus, $U_j=U_{0,j}$.

\begin{lem}\label{omega-psi-lem}\label{foliated-dr-lem}
Under hypothesis (*) above, we have
\eqspl{omega-psi}{
\H^i(\Omega^\bullet_\psi)=
\begin{cases}
i_{U_1!}(\O^\psi(-D)\psi), i=0;\\
0, i>0.
\end{cases}
}
\end{lem}
\begin{proof}
To begin with, the RHS of \eqref{omega-psi} clearly maps naturally
to the LHS, so it suffices to prove that this map is an isomorphism
locally at each point.\par
Consider first the elementary case of a point $p\in U$. There the
 quotient complex $\Omega^\bullet_{X/\psi}$ is the usual relative
 De Rham for the foliation determined by $\psi$, which is a resolution
 of $\O_X^\psi$. Then the cohomology sequence
 of \eqref{quotient}
 reduces to
 \[\exseq{\C_X}{\O^\psi}{\O_X.\psi},\]
 the second map being exterior derivative,
so  we get the result. The case $p\in U_1$ is
 similar, because there we may assume $\psi=dx/x, F=x$ so $F\psi=dx$
 like before.\par
 
 Now we may assume $p\in D_1$, double locus of $D$, and show $\Omega^\bullet_\psi$
 is exact.We will use hypothesis (*), which says that $\psi$ and 
 $\chi:=\chi_{(m.)}+\chi_{(1.)}$
 are linearly independent at $p, \forall (m.)$. For simplicity we assume $p$
 is a point of maximum multiplicity, i.e. $d$, on $D$. The closed log 1-form
 $\psi$ can be written in the form
 \[\psi=\sum a_i\dlog{(x_i)}+ dg\]
 with $a_i$ constant and $g$ holomorphic. Then replacing $x_1$ by
 $\exp(g/a_1)x_1$, we may assume $g=0$ In particular, $\psi$ is homogeneous
 of degree 0.\par 

Consider the $E_1$ 'termwise-to-total' spectral sequence associated
to the resolution \eqref{resolution}. Each resolving term only has $\H^{d-1}$
and that is given by the appropriate $K^i$. Therefore the entire
$E_1$ page reduces to the following complex (occurring at height
$d-1$ in the second quadrant)
\eqspl{K}{K^1\to K^2\to...\to K^d\to 0}
where the maps are multiplication by $\psi$.
We claim that the larger complex
\[0\to K^0\to K^1\to...\to K^d\to 0\]
is exact. 
%
%
Now working on a given homogeneous component $S^\bullet_{(m.)}$,
$d^\circ$ itself is multiplication by $\chi_{(m.)}+\chi_{(0.)}$ By 
Assumption (*), the latter section together with $\psi$
forms part of a basis of $S^1_{(0.)}$. Therefore clearly multiplication
by $\psi$, which preserves multidegree,
 is exact on the kernel (= image) of multiplication by
$\chi_{(m.)}+\chi_{(0.)}$, i.e. $K^.$. Therefore the larger complex extending \eqref{K}
is exact.

Thus, the $E_2=E_\infty$ 
page for the complex \eqref{K} just reduces to the $K^0$, sitting in bidegree $(-d, d)$, which yields 
our claim.
\end{proof}
As an immediate consequence of Lemma \ref{omega-psi-lem}, we conclude
\begin{cor}
Under hypothesis (*), we have
\[\H^i(\Omega_{X/\psi}^\bullet)=\H^i(\Omega^\bullet_\psi), i>0,\]
and there is an exact sequence
\eqspl{omega/psi}{
\exseq{\C_X}{\H^0(\Omega_{X/\psi}^\bullet)}{i_{U_1!}(\O^\psi(-D)\psi)}.
}
\end{cor}
\begin{proof}
Using the long cohomology sequence of
 \eqspl{quotient}{\exseq{\Omega^\bullet_\psi[-1]}{\Omega^\bullet_X}
 {\Omega^\bullet_{X/\psi}}
 ,}
the assertion follows from Lemma \ref{omega-psi-lem}.\end{proof}
We will require a generalization of Lemma \ref{omega-psi-lem} to a $k$-tuple 
of forms. Thus, with notations as above, let $\psi_1, ...,\psi_k$
be sufficiently general closed log  1-forms on $X$ and set
\[\psi_{\ul{k}}=\psi_1...\psi_k, \Omega^\bullet_{\psi_{\ul{k}}}=\psi_{\ul{k}}\Omega^\bullet_X\llogmp{D}.\]
Let $\O^{\psi_{\ul{k}}}$ be the sheaf of holomorphic functions $f$ such that
$df\equiv 0\mod \psi_{\ul{k}}$ (i.e. such that $df$ is in the $\O_X$-module generated
by $\psi_1,...,\psi_k$).
\begin{lem}\label{omega-psi-k-lem}
Notations as above, we have
\eqspl{}{
\H^i(\Omega^\bullet_{\psi_{\ul{k}}})=\begin{cases}
i_{U_k!}(\O^{\psi_{\ul{k}}}(-D)\psi_{\ul{k}}), i=0\\
0, i>0.
\end{cases}
}
\end{lem}
\begin{proof}
The proof is by induction on $k$. The induction step is analogous to the proof of Lemma
\ref{omega-psi-lem}, with $\Omega^\bullet_{\psi_{\ul{k-1}}}$ replacing 
$\Omega^\bullet_X\llogm{D}$.
\end{proof}

We shall also need an analogue of Lemma \ref{omega-psi-lem} for the 
minusplus complex. Thus
set (cf. \S \ref{logmp-sub-sec})
\[\Omega^\bullet_{\psi+}=\psi\Omega^\bullet_X\llogmp{D}.\]
This is a complex that starts with $\psi\O_X(-D)$ in degree 0.
Set
\[\Omega^{i,\psi}=\{\alpha\in \Omega^i_{X_i}\llog{D_i}:d\alpha\equiv 0\mod\psi\}
/\psi\Omega^{i-1}_X\llog{D}.\]
Thus by definition, $\psi\Omega^{i,\psi}$ consists of the closed forms 
in $\psi\Omega^{i}_{X_i}\llog{D_i}$. Similarly with $\psi$ replaced by $\psi_{\ul{k}}$.
These complexes admit a natural increasing filtration with quotients
$\Omega_{X_k, \psi_\uk}^{[k}.$
\begin{lem}\label{psi-logmp-hom-lem}
Hypotheses as above, we have
\eqspl{omega-psi-plus}{
\H^i(\Omega^\bullet_{\psi+})=i_{U_{i,i+1}!}(\psi\Omega^{i,\psi}_{X_i}(-D_i))} 
and more generally
\eqspl{omega-psi-k-plus}{
\H^i(\Omega^\bullet_{\psi_\uk+})=i_{U_{i,i+k}!}(\psi\Omega^{i,\psi_\uk}_{X_i}(-D_i))} 

\end{lem}
\begin{proof}
Follows from Lemma \ref{omega-psi-lem}
and Lemma \ref{omega-psi-k-lem}, using the spectral sequence for a filtered complex,
which degenerates at $E_2$ for support reasons.
\end{proof}
\subsection{Simplicial 
De Rham complex on normal crossing varieties}\label{simplicial de rham sec}
Let $D$ be a variety with local normal crossings. Thus, $D$ is locally
embeddable as a divisor in a manifold $X$ with defining equation
$F=x_1...x_m$, where $x_1,...,x_m$ are part of a local coordinate
system (we call these 'adapted' coordinates). Let $X_k$ be the 
normalization of the $k$-fold locus of $D$, with (unramified) natural map
\[p_k:X_k\to D.\]
Thus, a point in $X_k$ is specified by a $k$-tuple $I\subset\mathbf{m}$
plus a point where $x_i=0, \forall i\in I$, and $X_k$ is a transverse
union of smooth branches $X_I$ corresponding to choices of $I$. 
So $X_1$ is just the normalization of $D$ and $X_k$ generally
is the normalization of the $k$-fold locus of $D$.
Note that there is a natural map
\[\rho_k:p_{k*}\Omega^\bullet_{X_k}\to p_{k+1*}\Omega^\bullet_{X_{k+1}}\]
defined by, for any $(k+1)$-tuple $I=(i_1<...<i_{k+1})$, 
\[\rho(\omega)_I=\sum\limits_{j=1}^{k+1} (-1)^j\omega_
{I\setminus i_j}\]
where $\omega_J$ is the restriction of $\omega$ on the branch-intersection
$X_J:=(x_j:j\in J)$.
It is easy to check that this is a morphism of complexes and that
$\rho_{k+1}\circ\rho_k=0$ so we get a double complex 
$(\Omega^\bullet
_{X_{\bullet}}, d, \rho)$, which we will call the \emph{simplicial
De Rham complex} associated to $D$ or \emph{simplicial
De Rham resolution of $\Omega^\bullet_D$} (see below).\par
On the other hand, recall that we have a complex- actually dg algebra,
namely  $\Omega^\bullet_D$,
which is the quotient of $\Omega^\bullet_X$ by the exterior
ideal generated by $\O(-D)$ and its image by $d$ in $\Omega^1_X$,
i.e. locally by $F=x_1...x_m$ and 
$dF=F\sum dx_i/x_i$. The following result is probably
well known.
\begin{lem}\label{simplicial-derham-lem}
(i)
$\Omega^\bullet_{X_\bullet}$ is a resolution of $\Omega^\bullet_D$\ \ .\\
(ii) There is an exact sequence
\eqspl{}{
\exseq{\Omega^\bullet_X\llogm{D}}{\Omega^\bullet_X}{\Omega^\bullet_D}.
}
\end{lem}
\begin{proof}(i) 
To begin with, there is clearly a map
 $\Omega^\bullet_D\to\Omega^\bullet_{X_\bullet}$ and it is easy
 to check locally that this map induces an isomorphism 
$\Omega^\bullet_D\to\ker(\rho_0)$. It remains to prove that
$\ker(\rho_{k+1})=\im(\rho_k)$ which can also be done locally,
so we can choose a local basis. We may
assume each constituent $\omega_I$ is extended over $D_i$
for all $i\in I$ via some compatible collection of deformation retractions
$X_i\to X_I$. Then the required exactness follows by using the
following homotopy operator
\[h_k:p_{k*}\Omega^\bullet_{X_k}\to 
p_{k-1*}\Omega^\bullet_{X_{k-1}}\]
\[(h_k(\omega))_J:=\sum\limits_{i\not\in J}\sgn(J|i)\omega_{J\cup i},\]
where
\[\sgn(J|i)=(-1)^{|\{r:j_r>i\}|}.\]
(ii) It is easy to check locally that the image of the pullback map
\[\Omega^i_X\to p_{1*}\Omega^{i}_{X_1}\]
coincides with the kernel of $\rho_1$, and then that the kernel of the same
pullback map coincides with $\Omega^i_X\llogm{D}$.
\end{proof}
\section{A double helix}
In what follows we will fix a manifold $X$ of even dimension $2n$ endowed with a Poisson
structure $\Pi$. Our main interest is in the case where $(X, \Pi)$ is holomorphic,
i.e. $X$ is a complex manifold (of complex dimension $2n$), and $\Pi$ is holomorphic,
and especially where  $\Pi$ is pseudo-symplectic in the sense 
that on some dense open subset of $X$, $\Pi$ is nondegenerate,
i.e. dual to a symplectic structure.
However, some of the basic constructions apply without the pseudo- symplectic condition 
and in the real $C^\infty$ case as well.
\par
Brylinski \cite{brylinski} constructed on the sheaves of differential forms on $X$
the structure of a 'reverse de Rham' complex
\[ ...\Omega^i_X\stackrel{\del}{\to}\Omega^{i-1}_X\to..., \ \ \del=d\iota_\Pi-\iota_\Pi d.\]
where $\iota_\Pi$ denotes interior multiplication by $\Pi$ (which lowers degree by 2)
and $d$ is the usual exterior derivative. In fact, his construction 
is quite formal and is valid 
generally in the context of a Poisson structure on a ringed space $X/B$ where
$B$ is a ringed space over $\Q$,
interpreted as a linear map $\Omega^2_{X/B}\to\O_X$, and where $\iota_\Pi$
is the natural extension of the latter to a degree- (-2) map on $\Omega^\bullet_X$.\par
Our observation here is first that there exists a \emph{different} reverse de Rham complex,
which we call the 'Mahr de Poisson' (MdP) complex $\Theta^\bullet_X$
with differentials not proportional to Brylinski's, valid in similar generality.
An essential feature of the MdP complex not shared by
the Brylinski complex  is the existence of a 'bonding map'
\[\pi:\Theta^{n]}_X\to\Omega^{n]}_X\]
as well as a dual bonding map
\[\pi':\Omega^{[n}_X\to\Theta^{[n}_X.\]
Based in this, we will
define a pair of hybrid complexes
$\ed^\bullet, \de^\bullet$, each of type 'half de Rham, half twisted reverse de Rham'. 
We will then construct a map between
these complexes and study its mapping cone, identifying it with an
analogous complex built on the \emph{foliated} de Rham/ twisted reverse de Rham
complex associated to a 'degeneracy foliation' defined on the degeneracy 
or Pfaffian divisor of $\Pi$. 'Morally speaking', it is the existence of this foliation and its
associated foliated de Rham complex, which is a quotient of the de Rham complex
of $X$, that force our twisted reverse de Rham complex to exist, essentially
as the kernel of the quotient map. See \cite{ciccoli} or \cite{dufour-zung} for basic
information on Poisson structures.
\par
To begin with, define an operator
\eqspl{}{
	\delta:\Omega^{n+i}_X\to\Omega_X^{n+i-1}, i\in [-n,n],\\
	\delta=id\iota_\Pi-(i-1)\iota_\Pi d.
} 
[To simplify the notation we will sometimes suppress the
interior multiplication symbol and simply write this operator
as $id\Pi-(i-1)\Pi d$. We will also denote $\iota_{\Pi}(\omega)$
by $\carets{\Pi, \omega}$].
Here $n$ is of course half the dimension of $X$ if $X$ is a
Poisson manifold of dimension $2n$, 
or just an arbitrary natural number if $X/B$ is an
arbitrary Poisson ringed space (in which case the construction 
will of course depend on $n$).
Note that this differential  is not proportional to Brylinski's.
Then define sheaves $\ed^i, \de^i, \Theta^i$ by
\eqspl{}{
	\Theta^i&=\Omega_X^{2n-i}, i\in[0, 2n],\\
	\ed^i&=\begin{cases} \Omega^{n-i}_X, i\in[-n, 0]\\
		\Omega^{n+i}_X, i\in [1,n],
	\end{cases}\\
	\de^i&=\begin{cases}\Omega^{n+i}_X, i\in [-n,0]\\
		\Omega^{n-i}, i\in [1,n].
	\end{cases}
}
Note that the maps  defined by interior multiplication
\eqspl{}{
	\iota_{\Pi^k}:\Omega^{n+k}\to \Omega^{n-k}
}
yield for each
$i\in [0,n]$ a map
\eqspl{}{
	\pi:\Theta^i\to \Omega^i
}
and for each $i\in [n,2n]$ a map
\eqspl{}{
	\pi':\Omega^i\to\Theta^i
} and also for each
$i\in[-n,n]$ a map
\eqspl{pi-on-dihelical}{
	\pi=\iota_{\Pi^{|i|}}:\ed^i\to\de^i.
}
\begin{thm}\label{dihelical}
	Let $\Pi$ be a Poisson structure on a ringed space $X/B$.\par
	(i) Endowed with differential $\delta$,
	$\Theta^\bullet$ is a complex.\par
	(ii) Endowed with differential $\delta$ in negative degrees and $d$ in nonnegative degrees,
	$\de^\bullet$ is a complex.\par
	(iii) Endowed with differential $d$ in negative degrees and $\delta$ in nonnegative degrees,
	$\ed^\bullet$ is a complex.\par
	(iv) The map $\pi$ defined above yields morphisms of complexes
	\eqspl{}{
		\pi:\Theta^\bullet_{n]}\to\Omega^\bullet_{n]},\\
		\pi':\Omega^\bullet_{[n}\to\Theta^\bullet_{[n},\\
		\pi:\ed^\bullet\to\de^\bullet
	}where $n], [n$ denote truncation above (resp. below) degree $n$,
	which are isomorphisms locally wherever $\Pi$ is nondegenerate.
\end{thm}
The top and bottom squares of $\pi$ are, respectively:
\eqspl{}{
	\begin{matrix}
		&\Omega^{2n}_X&\stackrel{\iota_{\Pi^n}}{\to}&\O_X&\\
		\delta_n=nd\iota_\Pi-(n-1)\iota_\Pi d=nd\iota_\Pi&\downarrow&&\downarrow d&\\
		&\Omega^{2n-1}_X&\stackrel{\iota_{\Pi^{n-1}}}{\to}&\Omega_X^1&\\ \\ \\
		&\Omega^{2n-1}_X&\stackrel{\iota_{\Pi^{n-1}}}{\to}&\Omega^1_X&\\
		&d\downarrow&&\downarrow&\delta_{-n+1}=-(n-1)d\iota_\Pi+n\iota_\Pi d=n\iota_\Pi d\\
		&\Omega^{2n}_X&\stackrel{\iota_{\Pi^{n}}}{\to}&\O_X&
	\end{matrix}
} 
where $\delta_n=nd\iota_{\Pi}$ on $\Omega^{2n}_X$ and 
$\delta_{-n+1}=n\iota_\Pi d$ on $\Omega^1_X$.
The middle two squares are:
\eqspl{}{
	\begin{matrix}
		&\Omega^{n+1}_X&\stackrel{\iota_\Pi}{\to}&\Omega_X^{n-1}&\\
		d\iota_\Pi&\downarrow&&\downarrow d&&\\
		&\Omega^n_X&=&\Omega^n_X&&\\
		&\downarrow d&&\downarrow&-\iota_\Pi d\\
		&\Omega_X^{n+1}&\stackrel{\iota_\Pi}{\to}&\Omega_X^{n-1}.&
	\end{matrix}
} We will call the mapping cone of $\pi$ the \emph{dihelical} (double helix) complex
associated to $\Pi$.
\begin{cor}
	Locally where $\Pi$ is nondegenerate, $\ed^\bullet$ and $\de^\bullet$ are exact in degrees $\neq -n$.
\end{cor}
\begin{proof}
	Follows from exactness in positive degrees of the de Rham complex.
\end{proof}
\begin{cor}\label{T-D} Assume $(X, \Pi)$ is holomorphic and $\Pi$ is generically
	nondegenerate with Pfaffian divisor $D=[\Pi^n]$. Then
	$\Theta^\bullet_{n]}$ is isomorphic to the 'augmented twisted
	truncated Poisson complex'
	\[\O(-D)\to T_X(-D)\to T^2_X(-D)...\]
	where the first map sends a function $f$ to the corresponding
	Hamiltonian vector field $\carets{df, \Pi}$ and other maps are
	$[., \Pi]$.\par
\end{cor}
\begin{proof}
	There is a map
	\[\carets{., \Pi^{n-i}}:\Omega^{n-i}\to T^{n-i}\]
	coming from a morphism of complexes, 
	such that the composition $\Omega^{n+i}\to T^{n-i}$ is interior
	multiplication by $\Pi^n$, which is $F$ times a volume form where
	$F$ is an equation of $D$. This composite yields the 
	desired isomorphism
	of $\Theta_{n]}$ with the augmented twisted truncated Poisson 
	complex (we use 'augmented' because the usual Poisson complex start with
	$T_X$ and 'twisted' because the terms are (even though the differential is not).\end{proof}
{\textit{Remark.}} The referee points out that the contraction map above is well know and appears,
e.g. in  “Twisted Poincar\'e duality
between Poisson homology and cohomology”, Luo-Wang-Wu, J.
Algebra 442 (2015), 484-505.\par


Thus, the complex $\Theta^\bullet$ is not really 'new' but its realization
in terms of differential forms makes possible a useful connection
with Hodge theory (see \S\ref{hodge} below).\par
The proof of the Theorem uses the following Calculus lemma (for which which
the integrability condition $[\Pi,\Pi]=0$ is essential):
\begin{lem}
	We have
	\eqspl{commutation1}
	{d\iota_{\Pi^m}=\iota_{\Pi^{m-1}}(md\iota_\Pi-(m-1)\iota_\Pi d)}
	\eqspl{commutation2}{
		\iota_{\Pi^m}d=(m\iota_\Pi-(m-1)d\iota_\Pi)\iota_{\Pi^{m-1}}.
	}
\end{lem}
\begin{proof}[Proof of Lemma]
	To prove \eqref{commutation1} for $m=2$ we can use a direct local computation.
	In a slightly more canonical vein, we may compute, for any
	differential form $\phi$, by definition of Lie derivative
	\[L_\Pi\carets{\Pi,\phi}=d\carets{\Pi^2,\phi}-\carets{\Pi, d\carets{\Pi, \phi}}.\]
	On the other hand by the derivation property of Lie derivative
	and the fact that $L_\Pi\Pi=[\Pi,\Pi]=0$, we have
	\[L_\Pi\carets{\Pi,\phi}=\carets{\Pi,L_\Pi\phi}=
	\carets{\Pi, \carets{ d\carets {\Pi, \phi}}}-\carets{\Pi^2, d\phi}.\]
	Comparing the last two displayed equations yields \eqref{commutation1} for $m=2$,
	and the general case follows inductively. \eqref{commutation2} is proved similarly.
\end{proof}
\begin{rem*}
	Alternatively in the pseudo-symplectic case, the only case we need here, it suffices to
	prove the identities \eqref{commutation1}, \eqref{commutation2} on the dense
	open set where $\Pi$ is symplectic, which can be done by a simple local calculation.
\end{rem*}
\begin{proof}[Proof of Theorem] To prove that $\ed^\bullet$ and $\de^\bullet$
	are complexes, we start with the well-known relation (equivalent
	to vanishing of the square of the differential in Brylinski's complex):
	\[(d\iota_\Pi-\iota_\Pi d)^2=0.\]
	Expanding, and using $d^2=0$ and \eqref{commutation1} for $m=2$ yields
	\[d\iota_\Pi d\iota_\Pi=\iota_\Pi d\iota_\Pi d.\]
	Then a direct computation yields
	\eqsp{
		((i-1)d\iota_\Pi-(i-2)\iota_\Pi d)(id\iota_\Pi-(i-1)\iota_\Pi d)=\\
		i(i-1)(d\iota_\Pi d\iota_\Pi-\iota_\Pi d\iota_\Pi d)=0
	} Thus, $\delta^2=0$. Together with $d^2=0$ and some trivial
	verifications around the midpoint $i=0$, 
	this suffices to show that $\Theta^\bullet, \ed^\bullet$
	and $\de^\bullet$ are complexes.\par
	Finally, the proof the $\pi$ is a morphism of complexes
	amounts to commutativity of suitable squares and
	translates exactly to \eqref{commutation1} and \eqref{commutation2}.
\end{proof}

\section{Log-symplectic manifolds}
\subsection{log-symplectic form}\label{log-symplectic-form-sec}
A Poisson manifold $(X, \Pi)$ of even dimension $2n$ such that
the degeneracy divisor $D=D(\Pi)=[\Pi^n]$ has local normal crossings
is said to be \emph{log-symplectic}.
The Poisson structure $\Pi$ can equivalently be described via
a 'log-symplectic form' $\Phi$.
 This is the meromorphic (in fact, logarithmic)  form defined by
\[\carets{\Pi^n,\Phi}=\Pi^{n-1}.\] 
Note that 
\[\carets{\Pi^n, \Phi^i}=\Pi^{n-i}.\]
Also, the maps on meromorphic forms
\[\Omega^{2n-i}_{X, mero}\stackrel{\carets{.,\Pi^{n-i}}}{\to}
\Omega_{X, mero}^i,
\Omega_{X, mero}^i\stackrel{.\wedge\Phi^{n-i}}{\to}
\Omega^{2n-i}_{X, mero}
\] are inverse to each other. We can write
\[\Pi^n=FV, \Phi^n=F\inv V^*\]
where $V, V^*$ are dual generators of $T^{2n}_X, \Omega^{2n}_X$
and $F$ is an equation for $D$.
Thus $\carets{V^*, \Pi^{n-1}}=F\Phi$ and for any $v\in \wedge^iT_X$
we have
\eqspl{pi-phi}{\carets{\carets{V^*, v}, \Pi^{n-i}}=\carets{\carets{V^*, \Pi^{n-i}}, v}
=F\carets{\Phi^i,v},}
thus the two maps
\eqspl{}{
\carets{.,\Pi^{n-i}}:\Omega^{2n-i}_X\to \Omega^i_X, \ \ 
F\carets{.,\Phi^i}:\wedge^iT_X\to\Omega^i_X
} are essentially the same under  the exterior duality identification
\[\Omega^{2n-i}_X(D)=\Omega_X^{2n-i}\otimes\wedge^{2n}T_X\simeq \wedge^iT_X\]
 and in particular they
have the same image.


\subsection{Log duality}
When  $\Pi$ is log symplectic with degeneracy divisor $D$, we have a 
'log-duality' map
\[\carets{\Pi, .}:\Omega^1_X\llog{D}\to T_X\mlog{D}.\]
This map is easily seen to be an isomorphism, with inverse
$\carets{\Phi, .}$, $\Phi$ being the corresponding log-symplectic form
(compare the proof of Proposition \ref{log-duality} below).
Another useful map, also called log-duality, is defined as follows.

Consider the map
\[\pi=\carets{., \Pi^i}:\Omega^{n+i}_X\to\Omega^{n-i}_X.\]
This clearly extends to a map, called the \emph{log duality} map,
\[\pi\llog{D}=\carets{., \Pi^i}:\Omega^{n+i}_X\llog{D}\to\Omega^{n-i}_X\llog{D}.\]
The following result was known before, see e.g. \cite{logplus}.
\begin{prop}\label{log-duality}
If $\Pi$ is log-symplectic, then $\pi\llog{D}$ is an isomorphism.
Moreover, $\Pi$ defines a nondegenerate alternating form on
$\Omega^1_X\llog{D}$.
\end{prop}
\begin{proof}
$\pi\llog{D}$ is clearly an isomorphism
locally off $D$ and
 at \emph{smooth} points of $D$. Thus, $\pi\llog{D}$ is
 a morphism of locally free sheaves
  of the same rank on $X$, which is an isomorphism
 off a codimension-2 subset, viz.the singular locus of $D$.
 Therefore $\pi\llog{D}$ is an isomorphism. The proof of the last
 assertion is similar, based on the fact that the (Pfaffian) degeneracy
 locus on $\Pi$ as alternating form on $\Omega^1_X\llog{D}$ is of pure
 codimension 1,  hence empty.
 \end{proof}
The remainder of this subsection extends duality to the case of 
the MdP complex $(\Theta^\bullet, \delta)$ defined in the last section
and uses the notations of that section.
 \begin{cor}\label{pipi-cor}
 Define a 'coduality' map of complexes of degree $2n$  as follows
 \eqspl{log-iso}{
 \pipi&:(\Theta\llog{D}^\bullet, \delta)\to (\Omega^\bullet\llog{P}, d),\\
 (\pipi)^i&=\iota_{\Pi^i}:\Omega^{n+i}\to\Omega^{n-i}, i=n,...,-n
 }
 where $\delta$ is the MdP differential and for $i<0$ , $\iota_{\Pi^i}$ means the inverse of the isomorphism
 $\iota_{\Pi^{|i|}}$, i.e. $\iota_{\Phi^{|i|}}$. Then $\pipi$ is an isomorphism of complexes.
 \end{cor}
 \begin{rem}
When $\Pi$ is $P$-normal,
the Proposition can also be proved by a straightforward local 
computation, using the normal form. Namely,
setting
\[d\log(x_I)=\bigwedge\limits_{i\in I}d\log(x_i), dy_J=\bigwedge\limits_{j\in J}dy_j,\]
and $(d\log(x_I)\wedge dy_J)\hat{}$ denoting the corresponding complementary multi-vectors,
we have
\[\iota_{\Pi^i}((d\log(x_I)\wedge dy_J)\hat{})=\pm dy_I\wedge d\log(x_J), |I|+|J|=i.\]
Thus, 
\[\iota_{\Pi^i}:\Omega_X^{2n-i}\llog{D}\to\Omega_X^i\llog{D}\]
sends a basis to a basis, hence is an isomorphism. \qed
\end{rem}

As a consequence, we can write down local generators for the cohomology
sheaves of $\Theta^\bullet\llog{D}$ for $\Pi$ P-normal
in terms of normal coordinates, cf. \eqref{normal-form}.
Fix a point $p$ of multiplicity $k$ on $P$ and a normal coordinate system $(x_i, y_i)$
so that precisely $x_1,...,x_k$ vanish at $p$. Set
\eqspl{}{
d_i=\dlog(x_i)dy_i, i=1,...,k, \\
d_i=dx_idy_i\  \mathrm{where}\  x_i\neq 0,\\
d_I=\bigwedge\limits_{i\in I}d_i, \ \ \dlog(x)_I=\bigwedge\limits_{i\in I} \dlog(x_i),
I\subset [1,k].
}
We recall (cf. e.g. \cite {deligne-hodge-ii} or \cite{grif-schmid}) that
the local cohomology of the usual log complex is well known
by Deligne, Griffiths and others. It is generated over $\C$  by the
 $\dlog(x)_I$ for various multi-indices $I\subset [1,k]$. 
 Note that the $x_i, i\in I$ are defining equations for a branch of 
$D^{|I|}$,  denoted $X_I$, and we have
\eqspl{log-local}{
\calh^i(\Omega_X\llog{D})=\begin{cases}
\nu_{i*} \C_{X_i}, i=0,...,n,\\
0, i>n.
\end{cases}
}
Applying the $\pipi$ isomorphism above, we conclude:
\begin{cor}\label{generators-cor}
Notations as above, $\Pi$ P-normal, we have
\eqspl{log-theta-local}{
\calh^{i}(\Theta^\bullet_X\llog{D})=\begin{cases}
\nu_{i*} \C_{X_i}, i=0,...,n,\\
0, i>n.
\end{cases}
}
If $\Pi$ is P-normal and $(x.)$ are normal coordinates,
then  the cohomology admits local generators of the form
\eqspl{}{
d_I\dlog(x)_J,\ \forall  I\coprod J=\{1,...,k\},
} 
where the latter generator is supported on the local
branch $X_J$
with equations $x_j, j\in J$.\par 

\end{cor}

\par
\subsection{ Standard form}\label{standard-form-sec}
We return to the case $\Pi$ arbitrary log-symplectic.
The $\pi\llog{D}$ isomorphism is useful in yielding a standard
form for $\Pi$ and the corresponding log-symplectic form $\Phi$, as follows. Let $F=x_1...x_m$ be a local equation for $D$ where
$x_1,...,x_{2n}$ are local coordinates. Set
\[v_i=\begin{cases}x_i\del_{x_i}, 1\leq i\leq m;\\
\del_{x_i}, m+1\leq i\leq 2n.
\end{cases}.\]
These form a local basis for the sheaf of log vector fields $T_X\mlog{D}$.
Let $v_i^*$ be the dual basis for $\Omega^1_X\llog{D}$ ($=dx_i/x_i$
or $dx_i$). Then
\eqspl{pi-form}{
\Pi=\sum a_{ij}v_iv_j,
}
\eqspl{phi-form}{
\Phi=\sum b_{ij}v^*_iv^*_j
} where $A=(a_{ij}), B=(b_{ij})=A\inv$ are skew-symmetric
and holomorphic. 
In fact, $B=\frac{1}{F}\wedge^{n-1}A$.
These are the matrices of the isomorphism $\pi\llog{D}$ and its
inverse. \par
\subsection{log (co)normal bundle}\label{log-normal-sec}
Here $\Pi$ is arbitrary log-symplectic.
\subsubsection{First order}
Notations as above, the natural map induced by inclusion
\[\nu_1^*(T_X\mlog{D})\to \nu_1^*(T_X)\]
has an $\O_{X_1}$-invertible kernel, denoted $N_{\log{(D)}}$, called
the \emph{log normal bundle} associated to the normal-crossing divisor $D$.
$N_{\log{(D)}}$ is dual to the cokernel of the inclusion
\[\nu_1^*(\Omega^1_X)\to\nu_1^*(\Omega^1_X\llog{D}),\]
hence via residue $N_{\log{(D)}}$ is globally free with
local generator $x_1\del_{x_1}$ where $x_1$ is a branch equation for $D$.
We have exact sequences
\eqspl{log-normal-seq}{
\exseq{N_{\log(D)}}{\nu^*_1(T_X\mlog{D})}{&T_{X_1}\mlog{D_1}}, \\
\exseq{&T_{X_1}\mlog{D_1}}{\nu_1^*(T_X)}{N_{X_1/X}}.}
When $D$ is the polar divisor of a log-symplectic
form $\Phi$ we denote by $\check N_{\log(D)}$ the
image of $N_{\log(D)}$ by the log duality map $\pi\llog{D}$.
This is a priori a  a line subbundle of
$\nu_1^*(\Omega^1_X\llog(D)),$ 
but in the exact residue sequence
\[\exseq{\Omega^1_{X_1}\llog{D_1}}{\nu_1^*(\Omega^1_X\llog{D})}{\O_{X_1}},\]
clearly the residue map, which is given
by interior multiplication by $v_1=x_1\del_{x_1}$,
 is zero on $\check N_{\log{D}}$, so
it is actually a line subbundle of $\Omega^1_{X_1}\llog{D_1}$.
\par
 \par

Note that unlike the usual conormal,
the log conormal is a \ul{sub}bundle of the log differentials on $X_1$,
and it is naturally isomorphic rather than dual to the log normal.
We get a canonical generator of $\check N_{\llog{D}}$, denoted $\psi_1$.
In terms of a standard form $\Phi=\sum b_{ij}\dlog(x_i)\dlog(x_j)$
as in \S\ref{standard-form-sec},
$\psi_1$  
has the form, locally on $X_1$ where $x_1$ is a branch equation
\eqspl{}{
\psi_1=\sum\limits_{i=2}^{2n}b_{1i}\dlog(x_i)
=\carets{\Phi, v_1}.
}

Note that $\psi_1$ is a \emph{closed} log form on $X_1$. It suffices to
check this at a general point of $X_1$,
where we may assume
(with a different coordinate system) that 
$\Phi=dx_1dx_2/x_1+\sum\limits_{i=2}^{2n} dx_{2i-1}dx_{2i}$,
$v_1=x_1\del_{x_1}$ so $\psi_1=dx_2$ is closed.
\subsubsection{Higher order}
Essentially the same construction applies to the higher-order loci $X_k$.
Thus, a point in $X_k$ comes equipped with $k$ transverse normal hyperplanes
corresponding to $k$ branches of $D$, which are well-defined up to order.
Hence the kernel of
\[\nu_k^*(T_X\mlog{D})\to\nu_k^*(T_X)\]
is a flat, integrable rank-$k$ bundle, 
denoted $N^k_{\log(D)}$, called the log normal
bundle of order $k$. It is locally generated by the log vector fields
$x_1\del_1,...,x_k\del_k$. Since these are canonical up to order, 
the log normal bundle becomes trivial after a suitable $S_k$-cover,
and is already trivial if $D$ has \emph{simple} normal crossings.  
We have exact sequences of locally free $\O_{X_k}$-
modules
\eqspl{log-normal-seq-order-k}{
\exseq{N^k_{\log(D)}}{\nu^*_k(T_X\mlog{D})}{&T_{X_k}\mlog{D_k}}, \\
\exseq{&T_{X_k}\mlog{D_k}}{\nu_k^*(T_X)}{N_{X_k/X}}.}

In the log-symplectic case,
$N^k_{\log(D)}$ is isomorphic via log duality
to a trivial
rank-$k$ subbundle of $\nu_k^*(\Omega^1_{X}\llog{D})$,
denoted $\check N^k_{\log(D)}$,
   with local generators
$\psi_i=\carets{\Phi, v_i}, i=1,...,k$.
Locally at a point, $X_k$ admits $k$ divisorial embeddings into
transverse branches of $X_{k-1}$, with associated log conormals
$\check N_i\subset\Omega^1_{X_k}\llog{D_k}$, respectively
generated by the $\psi_i$, 
and we have  $\check N^k_{\log(D)}=\bigoplus \check N_i$. 
As in the first-order case, we have $\check N_i\subset \Omega^1_{X_k}\llog{D_k}$,
 hence
$\check N^k_{\log(D)}\subset\Omega^1_{X_k}\llog{D_k}$.

%
%
Because the $\psi_i$ are closed forms ,
$\check N^k_{\log(D)}$ is an integrable subbundle, i.e.
corresponds to a codimension-$k$ foliation. This foliation
is known as the kernel or symplectic foliation, due to
the following
\begin{lem}\label{conormal-kernel-lem}
Outside the divisor $D_k\subset X_k$, the conormal bundle 
$\check N^k_{\log(D)}$ coincides with the kernel
of the Poisson structure induced on $X_k$ by $\Pi$.
\end{lem}
\begin{proof}
On $\nu_1^*(\Omega^1_X\llog{D})$, $\Pi$ induces a nondegenerate form,
and it pairs the $\psi_i$ with the conormal forms $dx_i/x_i$.
Therefore $\Pi$ yields a nondegenerate form on the kernel of the natural map
$\nu_1^*(\Omega^1_X\llog{D})/\check N^k_{\log(D)}\to \bigoplus \O dx_i/x_i$,
that is $\Omega^1_{X_k}\llog{D_k}/\check N^k_{\log(D)}$.
\end{proof}
\subsubsection{Conormal filtration}
The subbundle $\check N^k_{\log(D)}$ defines, in the
usual way, an increasing, length-$k$
filtration $\cF^\perp_\bullet$ on $\Omega^\bullet_{X_k}\llog{D_k}$,
called the \emph{conormal filtration} 
defined by
\[\cF_j^\perp\Omega^\bullet_{X_k}\llog{D}=
\wedge^{k-j+1}\check N^k_{\log(D)}\Omega^\bullet_{X_k}\llog{D}.
\]
Thus,
\[\cF^\perp_j\Omega^\bullet_{X_k}\llog{D_k}=\sum\limits_{|I|=k-j+1}
\psi_I\Omega^\bullet_{X_k}\llog{D_k}.\]

\subsection{The Residual Generality condition}\label{rg-sec}
The log-symplectic Poisson structure $\Pi$ is said to be \emph{residually general},
or to satisfy the \emph{RG condition}, if at every point $p$ of multiplicity
$m$ on the degeneracy divisor, and a standard form $\sum a_{ij}v_iv_j$
as above, the matrix $(a_{ij}(p):i,j\leq m)$ is a general skew-symmetric $m\times m$
matrix. This condition can be obviously rephrased in terms of the corresponding
log-symplectic form $\Phi$ to say that its polar part is general. The RG condition
is stronger, for any $t\leq m$,  than the '$t$-very-general condition introduced in
\cite{logplus}, Erratum, hence also than the original 'general position' condition employed in
 \cite{logplus}.
\par One consequence of the RG condition is that for any $i\leq m$,
the (closed) 1-form $\psi_i=\carets{\Phi, v_i}$ pulls back to a general log 1-form
on the branch $(x_i)$ of $D$ and in particular its polar divisor coincides exactly 
with the divisor on $(x_i)$ induced by $D$, defined by $\prod\limits_{j\neq i}x_i$.
Furthermore, any collection of $m$ or fewer elements among the $\psi_i$.
and the standard forms $\dlog(x_1),...,\dlog(x_m)$ are linearly independent,
i.e. are a basis for a locally free and cofree submodule.
Consequently, the pullback of any collection of $\psi_i$ to any
multiplicity locus $X_k$ are linearly independent..
\par Note that the RG condition excludes P-normality (unless $D$ is smooth): indeed
if $\Pi$ is P-normal then $\psi_i$ above has no poles at all.

\section{Degeneracy, kernel foliation}
From now on we restrict attention to the case of a complex pseudo-symplectic
Poisson manifold $(X, \Pi)$ of dimension $2n$. 
Then the degeneracy locus of $\Pi$ is a (Pfaffian)
divisor $P=[\Pi^n]\in|-K_X|$ (for this section, 
not necessarily with normal crossings).
It is well known that $\Pi$ descends to a (degenerate) Poisson
structure on the smooth part of $P$: this follows from the fact
that the kernel of $\Pi$ on $\Omega_X$ at a smooth point of $P$ 
contains the conormal line (cf. \S \ref{log-normal-sec} above
or \cite{qsymplectic}, proof of Prop. 10). 
Here we will expand on this. More precise results will be given 
in \S \ref{local-cohomology-sec}, under the hypothesis 
that $\Pi$ is log-symplectic and residually general.\par

Define sheaves $C^i$  via the exact sequence
\eqspl{C-def-eq}{
0\to\Omega^{2n-i}_X\stackrel{\iota_{\Pi^{n-i}}}{\to}\Omega^i_X\to C^i\to 0.
}
Thus, $C^0=\O_P, C^n=0$. Note that each $C^i$ is an $\O_P$-module.
Also, the degeneracy ideal $\I_{2k}$ defined by Lima-Pereira \cite{lima-pereira}
is none other than the $(2n-2k)$-th Fitting ideal of $C^1$.
Moreover, by Theorem \ref{dihelical}, there are exact diagrams
\eqspl{}{
\begin{matrix}
0\to&\Omega^{2n-i}_X&\stackrel{\iota_{\Pi^{n-i}}}{\to}&\Omega^i_X&\to& C^i&\to 0\\
& \downarrow \delta&& \downarrow d&& \downarrow d&\\
0\to&\Omega^{2n-i-1}_X&\stackrel{\iota_{\Pi^{n-i-1}}}{\to}&\Omega^{i+1}_X&\to& C^{i+1}&\to 0.
\end{matrix}
}
\eqspl{}{
\begin{matrix}
0\to&\Omega^{2n-i}_X&\stackrel{\iota_{\Pi^{n-i}}}{\to}&\Omega^i_X&\to& C^i&\to 0\\
& \downarrow  d&& \downarrow \delta&& \downarrow \delta&\\
0\to&\Omega^{2n-i+1}_X&\stackrel{\iota_{\Pi^{n-i+1}}}{\to}&\Omega^{i-1}_X&\to& C^{i-1}&\to 0.
\end{matrix}
}
Thus, we effectively get two mutually reverse complexes:
\eqspl{}{
(C^\bullet_{n]}, d): 
C^0=\O_P\stackrel{d}{\to}C^1\stackrel{d}{\to}
 C^2\stackrel{d}{\to}...\stackrel{d}{\to}C^{n-1}\to C^n=0,
}
\eqspl{}{
(C^\bullet_{[n}, \delta):C^n=0\to
C^{n-1}\stackrel{\delta}{\to}C^{n-2}\stackrel{\delta}{\to}...
\stackrel{\delta}{\to}C^1\stackrel{\delta}{\to} \O_P.
}
As to the interpretation of these, we have Proposition \ref{c-i} below. First, an auxiliary
 multilinear algebra result.
 \begin{lem} For $i\leq j\leq k$, there exist bilinear forms
 \[P_{i,j}(.,.,\Pi^i):\Omega^i_X\times\Omega_X^k\to\Omega_X^k\]
 (linear in $\Pi^j$ as well), such that
 \[\alpha\wedge\carets{\Pi^j, \beta}=\carets{\Pi^{j-i}, P_{i,j}(\alpha, \beta, \Pi^i)},
  \alpha\in\Omega^i_X, \beta\in\Omega^k_X.\]
 \end{lem}
 \begin{proof}
 It suffices to prove this for $\alpha$ completely decomposable, hence by induction we are reduced to
 the case $i=1$. There, 
 Note the following:
 \eqsp{\carets{\Pi^{j-1}, \carets{\Pi, \alpha\wedge\beta}}=\carets{\Pi^j, \alpha\wedge\beta}
 =\alpha\wedge \carets{\Pi^j, \beta}\pm j\carets{\carets{\Pi^{j-1}, \beta}, \carets{\Pi, \alpha}}\\
  =\alpha\wedge \carets{\Pi^i, \beta}\pm
  j\carets{\Pi^{j-1}, \carets{\beta, \carets{\Pi, \alpha}}}.
 }
 Thus, an explicit formula for $P_{1,j}$ is
 \[P_{1,j}(\alpha, \beta, \Pi)=
 \carets{\Pi, \alpha\wedge\beta}\pm j\carets{\carets{\Pi, \alpha}, \beta}
 \]
 \end{proof}
 \begin{cor}
 The image of the morphism $\pi:\Theta^\bullet_{n]}\to\Omega^\bullet_{n]}$  is an exterior
  ideal closed under $d$.
 Hence  $(C^\bullet_{n]}, d, \wedge)$ is a sheaf of differential
 graded algebras.
 \end{cor}
 Next, we compare the algebra $C^\bullet_{n]}$
 to the exterior algebra on $C^1$:
\begin{prop}\label{c-i}
(i) There is a canonical map
\[{\bigwedge^i} _{\O_P}C^1\to C^i.\]
(ii) At a smooth point of $P$, each $C^i$ is locally free over $\O_P$ and we
have an exact sequence
\eqspl{}{
\exseq{{\bigwedge\limits^{2n-i}}_{\O_P}C^1}{{\bigwedge\limits^{i}}_{\O_P}C^1}{C^i}
} 
where the first map is induced by $\iota_{\Pi^{n-i}}$.
\end{prop}
\begin{proof}
(i) results inductively  from the commutative diagram
\eqsp{
\begin{matrix}
\Omega_X^{2n-1}\otimes\Omega_X^{i}\oplus\Omega_X^{1}\otimes\Omega_X^{2n-i}
&\stackrel{\iota_{\Pi^{n-1}}\otimes\id\oplus\id\otimes\iota_{\Pi^{n-i}}}{\to}&
\Omega_X^1\otimes\Omega_X^i&\to& C^1\otimes C^i&\to 0\\
\downarrow&&\downarrow&&\downarrow&\\
\Omega_X^{2n-i-1}&\stackrel{\iota_{\Pi^{n-i-1}}}{\to}&
\Omega_X^{i+1}&\to&C^{i+1}&\to 0
\end{matrix}.
}
Here the left vertical map is $^tP_{i, n-1}(.,., \Pi^{n-i-1})\oplus P_{1, n-i}(.,.,\Pi^{n-i-1})$
where $^tP(\alpha, \beta, .)=P(\beta, \alpha, .)$, and the other vertical maps are just
wedge product.\par
(ii) At a smooth point of $P$, $\Pi$ admits a normal form
\[\Pi=x_1\del_{x_1}\del_{y_1}+\sum\limits_{i=2}^n\del_{x_i}\del_{y_i}.\]
From this, the assertion follows by an easy computation.
\end{proof}
\begin{cor}
$C^1$ is integrable and induces on the smooth part of $P$
a codimension-1 foliation by Poisson submanifolds, called the kernel foliation.
\end{cor}
\begin{proof}
Perhaps the easiest way to check the integrability condition
 is to use the normal form above, which shows
that at a smooth point of $P$, where $P$ has local equation $x_1$,
$C^1$ is the quotient of $\Omega_P$ by the subsheaf generated by $dy_1$,
and thus corresponds to the foliation by level sets of $y_1$.\par
The fact that $\Pi$ descends to the leaves of the foliation follows from the fact that
$\iota_\Pi:\Omega^1_X\to T_X$ vanishes over $P=[\Pi^n]$ on the image of 
$\iota_{\Pi^{n-1}}:\Omega_X^{2n-1}\to\Omega_X^1$.
Alternatively, this can also be proved easily using the normal form above.
\end{proof}
\begin{cor}
Over the smooth part of $P$, $C^1$ coincides with the quotient of $\Omega^1_P$
by the log-conormal bundle (cf. \S \ref{log-normal-sec}). 
\end{cor}
\begin{proof}
Follows from Lemma \ref{conormal-kernel-lem}.
\end{proof}
The existence of the kernel foliation is not a new result: this foliation
coincides with the so-called symplectic foliation associated to the
degenerate Poisson structure induced by $\Pi$ on $P$. See for instance
\cite{lima-pereira}.
\par
We will henceforth denote $C^\bullet_{n]}$ simply by $C^\bullet$.
\begin{rem}
There is a $\Pi$-trace map $\bigwedge\limits^2C^1\to\O_P$.
The composition $\bigwedge\limits^{2n-2}C^1\to \bigwedge\limits^2C^1\to\O_P$
is nowhere vanishing on the smooth part of $P$.
Therefore on the smooth part of $P$ we can also identify $C^2$ with the subsheaf
of traceless elements of $\bigwedge\limits^2C^1$.
\end{rem}
\section{P-normal case, examples}
We recall \cite{qsymplectic}, Proposition 7, that P-normal Poisson
structures  $\Pi$ can be characterized
by the existence of a local coordinate system
(called \emph{normal coordinates}) in which $\Pi$ has the form
\eqspl{normal-form}
{\Pi=\sum\limits_{i=1}^kx_i\del_{x_i}\del_{y_i}+\sum\limits_{i=k+1}^n\del_{x_i}\del_{y_i}.
}
In particular, $\Pi$ is log-symplectic.
\begin{example}[Modified Hilbert schemes]\label{hilb-example}
Let $S$ be a smooth surface endowed with a Poisson structure corresponding to 
a smooth anticanonical curve $D$. Then $\Pi$ induces a Poisson structure $\Pi\sbr n.$
on the Hilbert scheme $S\sbr n.$. The Pfaffian divisor $P$ corresponds to the subschemes 
having a nonempty intersection with $D$ and the kernel foliation has leaves
corresponding to subschemes having a fixed intersection point with $D$
so $D$ is the parameter curve and indeed, $D$ is elliptic.
Although $\Pi$ is not $P$-normal and $P$ does not have normal crossings,
$\Pi\sbr n.$ induces a P-normal Poisson structure $\Pi_X$  on the stratified
blow-up $X$ of the incidence stratification on $S\sbr n.$ (see \cite{incidence}).
The components of the Pfaffian divisor of $\Pi_X$ are birational to $D\spr i.\times S\spr n-i.\times\P^{i-1}, i=1,...,n$
and the kernel foliation on the latter corresponds to the map to $D$ defined by projection to $D\spr i.$
followed by the sum map $D\spr i.\to D$ coming from an addition law on the
elliptic curve $D$ (the addition law and the sum map depend on the choice of origin;
the fibres do not). This is the map
whose derivative is given by
\[(... ,\del_{y_1}, ...,\del_{y_i})\mapsto \del_{y_1}+...+\del_{y_i},\]
$y_i$ being induced by a coordinate $y$ on $D$.
Indeed a straightforward derivative calculation shows that at a general point 
of the latter component, which corresponds to a reduced point-scheme with
exactly $i$ points on $D$,
there are local coordinates such that $\Pi_X$ takes the form
\[u_1\del_{u_1}\del_{v_1}+\sum\limits_{i=2}^n\del_{u_i}\del_{v_i},\]
where $v_1$ is the coordinate on $D\spr i.$ corresponding to $y_1+...+y_i$.
\end{example}
\begin{example}[Toric Poisson structures]\label{toric}
Let $X=X(\Delta)$ be a smooth projective toric variety,
with torus $T\subset X$ acting on $X$ (cf. \cite{fulton-toric}). Thus $\Delta$ is a fan
in $N\otimes\R$ where $N=\Hom(\C^*, T)$
is the lattice of 1-parameter subgroups. Since $N_\C=
N\otimes\C$ is the (abelian) Lie algebra of $T$ and embeds into $H^0(T_X)$,
any element of $\wedge^2N_\C$ yields a Poisson structure on $X$.
These structures generically are log-symplectic,
with Pfaffian divisor $X\setminus T$,
but they are not P-normal. For $X=\P^n$ these structures 
are studied in \cite{lima-pereira}, where they are called \emph{diagonal}.
A general such structure in even dimension satisfies the Residual Generality
condition (see \S \ref{rg-sec}).\par
Now suppose that $\dim(X)=2n$ is even and that the fan $\Delta$
satisfies the following condition\par
(*) There is a basis $u_1, v_1, ..., u_n, v_n$ of $N$ such that for
any cone $\sigma\in\Delta$ and any $i=1,...,n$, either $u_i\not\in\sigma$
or $v_i\not\in\sigma$.
\par
For any $u\in N$ and $\sigma\in\Delta$, the limit at 0
of the 1-parameter subgroup $\C^*\to T$ corresponding to $u$
lies in the affine patch $X_\sigma\simeq\C^m\times(\C^*)^{n-m}$
 iff $u\in\sigma$.
Consequently, the assumption $u\not\in\sigma$ implies that as vector field, $u$
is nowhere vanishing on  $X_\sigma$,
while $u\in\sigma$ implies that $u$ on $X_\sigma$ is a log
vector field, of the form $x\del_x$. Thus,
condition (*) implies that the Poisson structure 
\eqspl{}{
\Pi=u_1v_1+...+u_nv_n
} is P-normal.\par
Regarding condition (*), note that, as pointed out by Jose Gonzalez, it can always 
be achieved by subdividing a given fan, which corresponds to replacing
a given toric variety by a toric blowup of itself. In particular, there exist  many
toric blowups of projective space with this property.
%
%
\end{example}
\begin{example}[Toric-by-torus structures]\label{toric2}
Let $Z$ be an $n$-dimensional smooth projective toric variety
with lattice $N$, and let $u_1,...,u_n$ be a basis for $N$, viewed
as vector fields.
Let $A$ be an $n$-dimensional complex torus and $t_1,...,t_n$
a basis for the constant vector fields on $A$. Then
\eqspl{}{
\Pi=u_1\wedge t_1+...+u_n\wedge t_n\in H^0(Z\times A, \wedge^2T_{Z\times A})
}
is clearly a P-normal Poisson structure on $X:=Z\times A$. The kernel foliation
on $X_i$ is generated by $t_1,...,t_i$, so it is 
generally not algebraically (or mermorphically) integrable.
\end{example}
It is worth noting that the \emph{smallest}
degeneracy locus, i.e. the zero-locus $P_n$, of a P-normal Poisson structure, has
itself a special structure:
\begin{prop}
Let $\Pi$ be a P-normal Poisson structure on a projective $2n$-manifold $X$ 
and $Y=P_n(\Pi)$ its zero locus.
If $Y\neq\emptyset$, then $Y$ admits a surjective map to a nontrivial abelian variety.
\end{prop}
\begin{proof} To begin with, it is well known that $Y$ is endowed with
a tangent vector field called Weinstein's modular field \cite{weinstein-modular}.
To construct this field 
directly in our case, and see that it is never zero, note that
$\Pi$ yields a canonical section of $\check N_Y\otimes\bigwedge\limits^2 T_X$.
By the normal form \eqref{normal-form}, $\Pi$ lifts to $\check N_Y\otimes T_X\otimes T_Y$,
because the defining equations of $Y$ are $x_1,...,x_n$, while $y_1,...,y_n$
are coordinates on $Y$.
There is a canonical map
\[\check N_Y\otimes T_X\otimes T_Y\to \check N_Y\otimes N_Y\otimes T_Y\to T_Y,\]
and again by the normal form \eqref{normal-form}, the image of $\Pi$ by the latter map is never zero
(with the notation of loc. cit. it has the form $\del_{y_1}+...+\del_{y_n}$).
Now use the following, probably well-known, result.
\end{proof}
\begin{lem}
Let $Y$ be a smooth projective variety endowed with a nowhere-vanishing vector field $v$. Then 
there is an Abelian variety $A$ and
a surjective map $ Y\to A$, such that $v$ descends to a nonzero constant vector field on  $A$.
\end{lem}
\begin{proof}
Consider the Albanese map ${\alb}:Y\to B=\mathrm{Alb}(Y)$.  By a result of Matsushima-Lichnerowicz-Lieberman (cf. \cite
{lieberman}, Thm. 1.5), $v$ induces a nonzero constant vector field on $B$, 
which of course preserves the image $\mathrm{alb}(Y)$. Consequently, $\mathrm{alb}(Y)$ is invariant under
a nontrivial abelian subvariety $A_1\subset B$. Let $A_2\subset B$ be a complementary abelian subvariety.
Thus, $A_1\to B/A_2=:A$ is an isogeny. Because $\mathrm{alb}(Y)$ contains $A_1$-orbits, the map
$Y\to A$ is clearly surjective, and $v$ descends to a nonzero constant vector field on $A$.
\end{proof}
\begin{rem}
In the above situation, it is not necessarily the case that 
$Y$ admits an action by an abelian variety. Let
$E$ be an elliptic curve, $L$ a nontorsion line bundle of degree 0, and
$Y=\P_E(L\oplus\O)$. For each $a\in E$, the translate of $L$ by $a$
is isomorphic to $L$, e.g. because $L$ can be defined by
constant transition functions.
Therefore the automorphism group $G$ of $Y$ fits in an exact sequence
\[1\to\gm\to G\to E\to 1\]
which induces an analogous exact sequence on tangent spaces.
Then, a nonzero tangent vector to $E$ lifts to a tangent vector to $G$,
which corresponds to a nowhere-vanishing vector field on $Y$; however,
$E$ does not act on $Y$ due to the nontriviality of $L$.
\end{rem}
\begin{example*}[Example \ref{hilb-example} cont'd]
In the Hilbert scheme example above, $Y=D\spr n.$, which maps to the elliptic curve $D$
by the sum map.
\end{example*}
\begin{example*}[Example \ref{toric2} cont'd]
In the toric-by-torus example above, 
$Y$ is a disjoint union of copies of the torus $A$.
\end{example*}
\section{Local cohomology of upper MdP complex}\label{local-cohomology-sec}
We now assume till further notice that our log-symplectic
Poisson structure satisfies the Residual Generality
condition, see \S \ref{rg-sec}.
Our aim is to study the MdP complex $\Theta^{n]}$ 
and its cohomology, first locally, then in the next section, globally.
We study $\Theta^{n]}$ locally via its image by the bonding map $\pi$,
and we study the latter  image in turn via the simplicial resolution as in 
\S\ref{simplicial de rham sec}. Thus, we denote by $I^\bullet_k$ 
or $I_k$ the pullback
of the image of $\pi$ to $X_k$. Otherwise, notations are as in \S\ref{log-symplectic-form-sec}. \par
\subsection{Image of bonding map via simplicial resolution}
 To begin with, note that by the discussion in \S
\ref{log-symplectic-form-sec}, the image of $\pi$ on $\Omega^{2n-r}_X$
coincides with $F\carets{\Phi^r,\wedge^rT_X}$. 
In particular it follows that $I_1^1$ is generated locally by the form 
$F_1\psi_1=F\carets{\Phi, \del_1}$ where $\del_i=\del_{x_i}$.
Next we will generalize this to higher-degree differentials and 
the higher strata $X_k$. 
Let $I_k$ denote the image of $\im(\pi)$ under the pullback
map on differentials attached to the map $X_k\to X$.
Working locally at a point
of $X_k$, we decompose the log-symplectic form $\Phi$ into its normal and
tangential components:
\[\Phi=\Phi_{\perp, k}+\Phi_{=,k}=\sum\limits_{j=1}^k\psi_jdx_j/x_j+\Phi_{=,k}\]
where $x_1,...,x_k$ are equations of the branches of $D$ at the point in
question and $\Phi_{=,k}$ is a log-symplectic form on $X_k$ itself.
Now the contraction of a log form of degree $a$ with a log polyvector field 
of degree $b\leq a$ is a log form of degree $a-b$ 
(and thus if $a=b$, a holomorphic function).
Hence note that  for any (resp. any log) polyvector field $u$,
$\carets{\Phi_{=,k}, u}$ is of the form $\alpha/x_e$ (resp. $\alpha$), where $\alpha$
is a log form on $X_k$ and $x_e$
is a factor of $F_k$. 
Note that an expression $F\carets{\Phi^r, u_1...u_r}$ can be nonzero
on $X_k$ only if the normal fields $\del_i=\del_{x_i}, i=1,...,k$
all occur among the $u_i$, so we may assume $u_i=x_i\del_i, i=1,...,k$.
 In that case, 
the only term in the binomial expansion of $\Phi^r$ that can contribute is
$\Phi_{k,=}^{r-k}\Phi_{\perp,k}^k$, which yields
\[\carets{\Phi^r,u_1...u_r}=\binom{r}{k}\carets{\Phi_{k,=}^{r-k}\psi_1...\psi_k, 
u_{k+1}...u_r}.\] 
The latter is a sum of terms where some number, say $a$ of the $u$-s are
contracted with $\psi$-s and the remaining $r-k-a$ are 
contracted with $\Phi_{k,=}^{r-k}$. Note that such a term is 
divisible by $\Phi_{k,=}^a$.
Thus we can write
\eqspl{}{
F\carets{\Phi^r, u_1...u_r}=\sum F_{k,I}\alpha_{I,s} \carets{\psi_1...\psi_k, w_J}\Phi_{=,k}^s
}
where the $w_J$ are suitable polyvector fields on $X_k$
 the $\alpha_I$
are suitable log forms, products of some $\carets{\Phi_{=,k}, u_\l}x_e$,
and $F_{k,I}=F_k/\prod\limits_{e\in I} x_e$ is the appropriate factor of $F_k$;
in all the terms appearing, we have $s\geq|J|$.
This can be rewritten as
\eqspl{im-pi-element}{
F\carets{\Phi^r, u_1...u_r}=\sum F_{k,I}\beta_{I,s} \psi_I\Phi_{=,k}^s
} where $\psi_I=\prod\limits_{i\in I}\psi_i$ and $s\geq k-|I|$. 
Due to the residual generality hypothesis on $\Phi$, the coefficients $\beta_I$ are
general log forms of their degree when the polyvector field $u_1...u_r$ is 
chosen generally.
Recalling the log-conormal filtration from \S \ref{log-normal-sec}, we conclude:
\begin{prop}
We have, where $\cF^\perp_\bullet$ denotes conormal filtration,
\eqspl{i-k-generators-prop}{
\I_k^\bullet&= \sum\limits_{I\subset\ul{k}, s} \psi_I\Phi^s_{=,k}
\Omega^\bullet_{X_k}\llogmp{D_k}[-|I|-2s]\\
&=\sum\limits_{s\geq  j-1}\Phi^s_{k,=}\cF^\perp_j\Omega^\bullet_{X_k}\llogmp{D_k}[-2s].
}
\end{prop}
\subsection{Cohomology}
Our goal is to compute the cohomology sheaves of $I_k^\bullet$ for fixed
$k$, and then
that of $I^\bullet$, via the simplicial resolution $I^\bullet_\bullet$. To this end, we 
note  first that an expression as in \eqref{im-pi-element} can be normalized. 
In fact, we may assume
that each $\beta_{I,s}$ with $I\neq\emptyset$, 
when written out in terms on a basis for 1-forms,
does not contain any term divisible by any $\psi_i, i\in I$ nor $\Phi_{=,k}$. 
In the first case the term is zero, while in the second case it can be added
to a term attached to $\Phi_{=,k}^{s+1}$.
With this proviso,
the expression \eqref{im-pi-element} is \emph{unique}.\par
Next, as in the proof of Lemma \ref{foliated-dr-lem},
we may assume that the $\psi_i$ and $\Phi$ have
constant coefficients, hence $I_k$ can be decomposed
into homogeneous components $S^i_{(m.)}$. Now consider a differential
$\gamma\in S^i_{(m.)}$ decomposed as in the proof of Lemma
\ref{omega-psi-lem}
and normalized as above. Suppose first
that the multiplicity $\mu_k$ of $D_k$ on $X_k$ at the point in question is
greater than $|I|$. Consider a nonzero
 term $F_{k,I}\beta_{I,s}\psi_I\Phi^s_{=,k}$ with
smallest $s$. Then
\[d(F_{k,I}\beta_{I,s}\psi_I\Phi^s_{=,k})
=F_{k,I}\beta_{I,s}\psi_I\Phi^s_{=,k}\wedge\chi_{(m.)+1_I}\]
where $1_I$ is the characteristic function of $I$. Due to the 
residual general position of the $\psi_i$, this cannot vanish
unless $\beta_{I,s}\wedge\chi_{(m.)+1_I}=0$, i.e. $\beta_{I, s}$
is divisible by $\chi_{(m.)+1_I}$. Proceeding inductively over $s$, the same
holds for all the $\beta$ coefficients, hence for $\gamma$. This proves exactness
of the complex $I_k$ locally over $X_k\setminus U_{k, \mu_k}$. 
\par
Now suppose
$\mu_k\leq |I|$.
Then it is easy to see that
\[F_{k,I}\psi_I=dx_I=\prod\limits_{i\in I}dx_i\]
which is a closed form.
Consider again a term $F_{k,I}\beta_{I,s}\psi_I\Phi^s_{=,k}$ with
smallest $s$. Then
\[d(F_{k,I}\beta_{I,s}\psi_I\Phi^s_{=,k})=d(\beta_{I,s}\prod\limits_{i\in I}dx_i\Phi^s_{=,k})
=\beta_{I,s}\prod\limits_{i\in I}dx_i\Phi^s_{=,k}\chi_{(m.+1_I)}.\]
If $d\gamma=0$, this vanishes.
But clearly this expression can vanish only if 
$\beta_{I,s}$ is a closed form modulo the coordinates in $I$,
i.e. $d\beta_{I,s}$ is in the complex generated by the $dx_i, i\in I$.
Thus, $\beta_{I, s}$ is a section of $\hat\Omega^{r, \psi_I}_{X_k}$
for some $r$.
Similarly, or inductively, for terms with higher $s$.
We have proven
\begin{prop}
The local cohomology sheaf of the pullback 
$I^\bullet_k$ of the complex $\im(\pi)$ on $X_k$ is as follows, where $\what\bullet$
indicates closed forms:
\large
\eqspl{im-on-xk}{
\H^i(I^\bullet_k)=\
\bigoplus\limits_{\stackrel{|I|+r+2s=i}{I\subset\ul{k}}} i_{U_{k, k+|I|}!}(\hat\Omega^{r, \psi_I}_{X_k}\psi_I\Phi_{=,k}^s)\\
=\bigoplus\limits_{\stackrel{t+2s=i, s\geq j-1}{I\subset\ul{k}}} i_{U_{k, k+|I|}!}(\what{\cF_j^\perp\Omega^t_{X_k}}\Phi_{=,k}^s)
}
\normalsize
\end{prop}
Now via the inclusion $I^\bullet_k=\im(\pi)_{X_k}\subset\Omega^\bullet_{X_k}$,
the complexes $I_k$ for varying $k$ form a double complex resolving
$\im(\pi)\simeq\Theta^\bullet$ (see the proof of 
Lemma \ref{simplicial-derham-lem}), so we study next the maps
$I_k\to I_{k+1}$ and their induced maps on cohomology.
Thus consider the middle cohomology of the short complex
\[H^i(I_{k-1})\to H^i(I_k)\to H^i(I_{k+1})\]
with terms given by \eqref{im-on-xk}, hence compactly supported
respectively over 
\[U_{k-1, k-1+|I|}, 
U_{k, k+|I|}, U_{k+1, k+1+|I|}.\] Over the common
intersection $U_{k+1, k+|I|}$, the complex is exact by the
argument of \S \ref{simplicial de rham sec}, the simplicial De Rham 
resolution. Over $U_{k,k}$, the left
map is clearly surjective (and the right term is zero). 
Over $U_{k+|I|, k|I|}$, the left term vanishes and
the middle and right terms consist of differentials on $k$-fold,
resp. $k+1$-fold branch intersections at a point of multiplicity
exactly $k+|I|$.Thus
 the kernel of the right
map consist of the differentials that descend from $X_k$ to $D_k$.
This also applies mutatis mutandis to the case $k=1$.  It follows
that the spectral sequence for the local cohomology of the double 
complex $I^\bullet_\bullet$ we have
\[E_2^{i,k}=\bigoplus\limits_{\stackrel{|I|+r+2s=i}{I\subset\ul{k}}}
 i_{U_{k+|I|, k+|I|}!}(\hat\Omega^{r, \psi_I}_{X_k}\psi_I\Phi_{=,k}^s)
.\]
We claim that this spectral sequence degenerates at $E_2$
Indeed consider an element of $E_2^{i,k}$ represented by the form
$a=\beta\prod dx_i\Phi^s_{=,k}$ on $U_{k+|I|, k+|I|}$.
The image of $a$ in $I_{k+1}$ can be written as $db$ for some $i-1$-form
$b$ with the same $I$ and $s$. Then the image $c$ of $b$ in $I_{k+2}$
 is exact for support reasons. But the class of $c$ is just
 the image of $a$ under the second-page differential $d_2^{i,k}$
 so that differential is zero. Likewise for further pages. Therefore the
 spectral sequence degenerates at $E_2$.
 Consequently we conclude
 \begin{prop}
 The local cohomology $\H^j(I^\bullet_\bullet)$ admits a filtration with
graded pieces 
\[\bigoplus\limits_{\stackrel{|I|+r+2s=i}{I\subset\ul{k}}}
  i_{U_{k+|I|, k+|I|}!}(\hat\Omega^{r, \psi_I}_{X_k}\psi_I\Phi_{=,k}^s)
  = \bigoplus\limits_{\stackrel{t+2s=i, s\geq j-1}{I\subset\ul{k}}}
    i_{U_{k+|I|, k+|I|}!}(\what{\cF_j^\perp\Omega^t_{X_k}}\Phi_{=,k}^s)
  \]
  for all $i,k$ with $i+k=j$.
 \end{prop}
This essentially computes the cohomology of the
upper MdP complex:

\begin{thm} Let $(X, \Pi)$ be a log-symplectic manifold 
of dimension $2n$ satisfying the Residual
Generality condition.
The local cohomology of $\Theta^{n]}_X$ is as follows : \par
$\H^0(\Theta^\bullet)=i_{U_0!}(\C_{U_0})$; \par for
$0<j<n$, $\H^j(\Theta^\bullet)$ has a filtration with graded pieces
\[\bigoplus\limits_{\stackrel{|I|+r+2s=i}{I\subset\ul{k}}}
 i_{U_{k+|I|, k+|I|}!}(\hat\Omega^{r, \psi_I}_{X_k}\psi_I\Phi_{=,k}^s)
  = \bigoplus\limits_{\stackrel{t+2s=i, s\geq j-1}{I\subset\ul{k}}}
     i_{U_{k+|I|, k+|I|}!}(\what{\cF_j^\perp\Omega^t_{X_k}}\Phi_{=,k}^s)
 \]
 for $i+k=j-1$.
\end{thm}
\begin{proof}
Let $K^\bullet$ be the kernel of the natural surjection
 $\Omega^\bullet_X\to\Omega^\bullet_D $.
Because $\Omega^\bullet_X$ and $\Omega^\bullet_D$ are
resolutions of the respective constant sheaves, we have a quasi-isomorphism
\[ K^\bullet\sim i_{U_0!}(\C_{U_0}).\]
Consequently we have
\[\H^0(\Theta^{n]})= i_{U_0!}(\C_{U_0}), \H^i(\Theta^\bullet)\simeq \H^i(I^\bullet),0< i<n.\]
Because $C^\bullet$ is a quotient of $\Omega^\bullet_D$ as we have seen,
the map $K^\bullet\to\Omega^\bullet_X$ factors through $\Theta^\bullet$ and
we have an exact diagram
\eqspl{}{
\begin{matrix}&&&0&&0&\\
&&&\downarrow&&\downarrow&\\
0\to&K^\bullet&\to&\Theta^{n]}&\to&I^\bullet&\to 0\\
&\parallel&&\downarrow&&\downarrow\\
0\to&K^\bullet&\to&\Omega^\bullet_X&\to&\Omega^\bullet_D&\to 0\\
&&&\downarrow&&\downarrow&\\
&&&C^\bullet&=&C^\bullet&\\
&&&\downarrow&&\downarrow&\\
&&&0&&0&
\end{matrix}
} 
Note that $K^0=\O_X(-D)$ maps isomorphically to $\Theta^0=\Omega^{2n}_X$,
so that $I^0=0$. Also, as we have seen in \S \ref{simplicial de rham sec},
 $\Omega^\bullet_D$ is quasi-isomorphic
to its simplicial resolution $\Omega^\bullet_{X_\bullet}$, which induces a simplicial
resolution $I^\bullet_\bullet$, also quasi-isomorphic to $I^\bullet$. Thus
\[\H^i(\Theta^\bullet=\H^i(I^\bullet_\bullet)=\H^{i-1}(\tilde C^\bullet),1<i <n,\]
and there is an exact sequence
\[\exseq{\C_D}{\H^0(\tilde C^\bullet)}{\H^1(\Theta^\bullet)}\]
Now the Theorem follows from the preceding Proposition.
\end{proof}
Now recall that for a smooth affine $d$-dimensional variety $Y$ and
a locally free coherent sheaf $\mathcal F$, the compact-support cohomology
vanishes:
\[H^i_c(Y, \mathcal F)=0, \forall  i<d\]
(this is because the compact-support cohomology is the limit
of local cohomology supported  at points, and the latter vanishes by depth considerations).
Now the sheaves occurring as summands in the Theorem are not themselves
coherent but via the De Rham complex they admit a resolution by locally free
$\O_{X_k}$ modules. Therefore each such summand on $X_k$ has vanishing $H^t$
for $t<\dim(X_k)=2n-2k$ provided $U_{k+|I|}$ is affine, hence in particular if
$X_{k+|I|}$ is Fano. Thus we conclude:
\begin{cor}\label{rg-consequence-vanishing-cor}
Suppose $(X, \Pi)$ satisfies the RG condition and moreover that $X_k$
is Fano for $k\leq a$. The $H^i(\Theta_X^\bullet)=0$ for $i\leq a$.
\end{cor}
As we shall see, when $X$ is K\"ahlerian the cohomology of $\Theta^\bullet_X$
can be computed in terms of the usual Hodge cohomology of $X$.

\section{K\"ahlerian case, Hodge cohomology}\label{hodge}
Here we assume that our holomorphic pseudo-symplectic  Poisson manifold
$(X, \Pi)$ is compact and K\"ahlerian (or more generally satisfies the $\del\bar\del$
 lemma), $\Pi$ otherwise arbitrary.
This has the usual implications vis-a-vis degeneration of spectral sequences
involving sheaves of holomorphic differentials
(see for instance \cite{grif-schmid}). Then similar results
can be derived for the MdP and dihelical complexes:
\begin{thm}
The global hypercohomology spectral sequences
\eqspl{}{
E_1^{p,q}=H^q(X, \Theta^p)\Rightarrow H^i(\Theta^\bullet),
}
\eqspl{}{
E_1^{p,q}=H^q(X, \ed^p)\Rightarrow H^i(\ed^\bullet),
}
\eqspl{}{ 
E_1^{p,q}=H^q(X, \de^p)\Rightarrow H^i(\de^\bullet),
} all degenerate at $E_1$.
\end{thm}
\begin{proof}
It suffices to prove this for $\Theta^\bullet$. Consider a class 
\[[\alpha]\in H^i(\Omega^j_X)=H^i(\Theta^{2n-j}_X)\]
represented by a harmonic $(j,i)$ form $\alpha$. Then
since $\del(\alpha)=0$,
$d_1(\alpha) \in H^i(\Omega_X^{j-1})$ is represented by
a multiple of $\del\carets{\Pi,\alpha}$, which is $\del$-exact
and $\dbar$-closed, hence, by the $\del\dbar$ lemma, also $\dbar$-exact,
i.e. null-cohomologous.
 Hence $d_1([\alpha])=0$. \par
 Next, write 
\[\delta(\alpha)=(j-n)\del\carets{\Pi,\alpha}=\bar{\del}(\beta)\]
for a $(j-1, i-1)$ form $\beta$. Then $d_2([\alpha])$
is represented by 
\[\delta(\beta)=(j-1-n)\del\carets{\Pi, \beta}-(j-2-n)\carets{\Pi, \del(\beta)}.\] 
Now
\[\dbar\del(\beta)=\del\dbar(\beta)=(j-n)\del^2\carets{\Pi, \alpha}=0.\]
Hence $\del(\beta)$ is $\dbar$-closed and $\del$-exact, hence $\dbar$-
exact. Since $\Pi$ is holomorphic, $\carets{\Pi, \del(\beta)}$ is also $\dbar$-exact,
hence null-cohomologous.\par
Next, note
\[\dbar\del\carets{\Pi, \beta}=\del\dbar\carets{\Pi, \beta}=
\del\carets{\Pi, \dbar\beta}=(n-j)\del\carets{\del\carets{\Pi, \alpha}}
=(n-j)\del\carets{\Pi, \delta(\alpha)}=(n-j)\delta^2(\alpha)=0\]
(the next to last equality due to $\del^2=0$).
Therefore $\del\carets{\Pi, \beta}$ is again $\dbar$-closed and $\del$-exact,
hence $\dbar$ exact, hence null-cohomologous. Thus, $d_2([\alpha])=0$.
The vanishing of the higher $d_r$ is proved similarly.

\end{proof}
\begin{cor}\label{theta-hom-cor}
(i) $H^i(\Theta^\bullet)$ admits a filtration with quotients
\[H^q(X, \Omega_X^{2n-i+q}), q=0,..., i, i=0,...,2n.\]  
(i)\  $H^i(\ed^\bullet)$ admits a filtration with quotients
\[H^{n+i-a}(X, \Omega_X^{n+|n-a|}), a=0,...,2n.\]
(ii)\ $H^i(\de^\bullet)$ admits a filtration with quotients
\[H^{n+i-a}(X, \Omega_X^{n-|n-a|}), a=0,...,2n.\]
\end{cor}
Thus, the cohomology of $\Theta^\bullet_X$ gives rise to a 'Poisson Hodge diamond'
with rows $H^{i, 2n}_X, ...H_X^{0, 2n-i}, i=0,...,2n$. 
This diamond is just the usual Hodge diamond of $X$
rotated clockwise by $90^\circ$.
\par Using Corollary \ref{rg-consequence-vanishing-cor}, we can now conclude
\begin{cor}\label{rg-consequence-hodge}
Assume $(X, \Pi)$ is a compact holomorphic K\"ahlerian log-symplectic manifold
 such that $\Pi$ satisfies the RG condition,
and that the normalized strata $X_k$ are Fano for $k\leq a$. 
Then the Hodge numbers of $X$ satisfy
\eqspl{}{
h^{2n-i, i}_X=0, i=0,...,a.
}
\end{cor}
This might be compared with the following result which not
strictly speaking a consequence of the foregoing but related to in
in that $\Theta^\bullet_X=T^\bullet_X(-D)$ (see Corollary \ref{T-D}).
It gives a source of unobstructed odd-dimensional Poisson manifolds.
\begin{prop}
	Let $(X, \Pi)$ be a compact holomorphic K\"ahlerian log-symplectic
$2n$- dimensional manifold such that the Hodge number $h_X^{1, 2n-2}=0$.
Then the normalized degeneracy locus $X_1$ together with the induced
Poisson structure $\Pi_1$ have unobstructed deformations
and those deformations lift to deformations of $(X, \Pi)$ inducing
locally trivial deformations of the degeneracy divisor.
	\end{prop}
\begin{proof}
	We have an exact sequence
	\[\exseq{T^\bullet_X(-D)}{T^\bullet_X\mlog{D}}{j_*T^\bullet_{X_1}}\]
	where $D$ is the degeneracy divisor of $\Pi$ and $j:X_1\to D\subset X$
	is the normalization map. \par 
	\emph{Claim}: the induced map
	\[\HH^1(T^\bullet_X\mlog{D})\to\HH^1(j_*T^\bullet_{X_1})\]
	is surjective.\par
	 Assuming this, a first-order deformation of $(X_1, \Pi_1)$
	given by $v\in \HH^1(j_*T^\bullet_{X_1})$
	lifts to a deformation $\tilde v\in \HH^1(T^\bullet_X\mlog{D})$,
	i.e. a deformation of $(X, \Pi)$ inducing a locally trivial deformation
	of $D$; as is well known (e.g. \cite{qsymplectic},  p. 1170 and \cite{logplus},
	Lemma 1), the latter deformations 
	are unobstructed thanks to Poisson duality and Hodge theory,
	hence the given first-order deformation $\tilde v$ 
	extends to a formal or analytic arc, hence the same is true of $v$.
	Consequently $(X_1, \Pi_1)$ has unobstructed deformations.\par
	\emph{Proof of claim:} This follows from the vanishing $\HH^2(T^\bullet_X(-D))=0$.
	To see the latter note the exact sequence
	\[...\ H^1(T^2_X(-D))\to \HH^2(T^\bullet_X(-D))\to H^2(T_X(-D))\ ...
	\]
	Because $D$ is anticanonical, the last group is Serre dual to
	$H^{1, 2n-2}_X$ which vanishes by hypothesis. Similarly
	the first group is dual to $H^{2, 2n-2}_X$ which also vanishes 
	thanks to Hodge symmetry $h^{1, 2n-2}=h^{2n-2, 1}=h^{2, 2n-1}$.
	\end{proof}
The condition $h^{1, 2n-2}_X=0$ seems weak and certainly holds for flag manifolds
and toric manifolds. Thus we conclude by Example \ref{toric} that the normalized
boundary of an even-dimensional toric variety $X$
carries unobstructed Poisson structures induced from $X$.

\bibliographystyle{amsplain}
\bibliography{../mybib}
\end{document}